\documentclass[a4paper,10pt,reqno]{amsart}

\usepackage[UKenglish]{babel}
\usepackage{amsmath}
\usepackage{amssymb}
\usepackage{amsthm}
\usepackage{verbatim}
\usepackage{stackrel}
\usepackage[arrow, matrix, curve]{xy} 

\usepackage[shortlabels]{enumitem}

\usepackage{comment,cite}	
\usepackage{hyperref}
\hypersetup{
    colorlinks=true,
    linkcolor=blue,
    citecolor=red,
    filecolor=magenta,      
    urlcolor=cyan,
    pdftitle={Persistent irreducibility},
    bookmarks=true,
    linktocpage=true,
}

\usepackage[utf8]{inputenc}

\usepackage{amssymb}
\usepackage{amsmath}
\usepackage{amsthm}
\usepackage{bbm}
\usepackage{verbatim, stackrel}

\usepackage[shortlabels]{enumitem}
\usepackage{marginnote}
\usepackage{color}

\usepackage{stmaryrd}
\usepackage{tikz}
\usepackage{tikz-cd}

\newcommand{\bbC}{\mathbb{C}}

\newcommand{\bbN}{\mathbb{N}}

\newcommand{\bbR}{\mathbb{R}}

\newcommand{\calL}{\mathcal{L}}

\newcommand{\calT}{\mathcal{T}}

\DeclareMathOperator{\id}{id} 
\DeclareMathOperator{\one}{{\mathbbm{1}}} 
\DeclareMathOperator{\re}{Re} 
\newcommand{\argument}{\mathord{\,\cdot\,}} 
\newcommand{\dx}{\;\mathrm{d}} 
\newcommand{\norm}[1]{\left\lVert #1 \right\rVert} 
\newcommand{\modulus}[1]{\left\lvert #1 \right\rvert} 
\newcommand{\duality}[2]{\left\langle#1\, ,\, #2\right\rangle} 
\newcommand{\dom}[1]{\operatorname{dom}\left(#1\right)} 
\DeclareMathOperator{\Ima}{Rg} 
\newcommand\restrict[1]{\raisebox{-.5ex}{$|$}_{#1}} 

\newcommand{\spec}{\sigma} 

\newcommand{\Res}{\mathcal{R}} 
\newcommand{\spr}{r} 
\newcommand{\spb}{s} 
\newcommand{\gbd}{\omega_0} 

\theoremstyle{definition}
\newtheorem{definition}{Definition}[section]
\newtheorem{remark}[definition]{Remark}

\newtheorem*{remark*}{Remark}
\newtheorem*{remarks*}{Remarks}
\newtheorem{example}[definition]{Example}

\theoremstyle{plain}
\newtheorem{proposition}[definition]{Proposition}
\newtheorem{lemma}[definition]{Lemma}
\newtheorem{theorem}[definition]{Theorem}
\newtheorem{corollary}[definition]{Corollary}

\newtheorem{conditions}[definition]{Conditions}

\numberwithin{equation}{section} 

\begin{document}

\title[Irreducibility of eventually positive semigroups]{Irreducibility of eventually positive semigroups}
\author{Sahiba Arora}
\address[S. Arora]{Technische Universität Dresden, Institut für Analysis, Fakultät für Mathematik , 01062 Dresden, Germany}
\email{s.arora-1@utwente.nl}
\author{Jochen Gl\"uck}
\address[J. Gl\"uck]{Bergische Universit\"at Wuppertal, Fakult\"at f\"ur Mathematik und Naturwissenschaften, Gaußstr.\ 20, 42119 Wuppertal, Germany}
\email{glueck@uni-wuppertal.de}
\date{\today}
\begin{abstract}
    Positive $C_0$-semigroups that occur in concrete applications are, more often than not, irreducible. Therefore a deep and extensive theory of irreducibility has been developed that includes characterizations, perturbation analysis, and spectral results.
	    Many arguments from this theory, however, break down if the semigroup is only eventually positive  -- a property that has recently been shown to occur in numerous concrete evolution equations. 
	
	    In this article, we develop new tools that also work for the eventually positive case. 
	    The lack of positivity for small times makes it necessary to consider ideals that might only be invariant for large times. 
	    In contrast to their classical counterparts -- the invariant ideals -- such eventually invariant ideals require more involved methods from the theory of operator ranges.
	    
	    Using those methods we are able to characterize (an appropriate adaptation of) irreducibility by means of linear functionals, derive a perturbation result, prove a number of spectral theorems, and analyze the interaction of irreducibility with analyticity, all in the eventually positive case. 
	    By a number of examples, we illustrate what kind of behaviour can and cannot be expected in this setting.
\end{abstract}

\maketitle

\section{Introduction}

\subsection*{Eventually positive $C_0$-semigroups}

A $C_0$-semigroup $(e^{tA})_{t\ge 0}$ on a function space or, more generally, on a Banach lattice $E$ is called \emph{individually eventually positive} if, for each $0\leq f\in E$, there exists a time $t_0\ge 0$ such that $e^{tA}f\ge 0$ for all $t\ge t_0$. If the time $t_0$ can be chosen independently of $f$, then the semigroup is said to be \emph{uniformly eventually positive}. 

Since Daners showed in \cite{Daners2014} that the semigroup $(e^{tA})_{t\ge 0}$ generated by the Di\-rich\-let-to-Neumann operator on the unit circle can, for suitable parameter choices, be eventually positive rather than positive, the theory of eventually positive semigroups garnered growing attention from the semigroup community. While the study of the finite-dimensional case has a somewhat longer history and can already be traced back to \cite{NoutsosTsatsomeros2008}, the development of the theory on general Banach lattices began with the papers \cite{DanersGlueckKennedy2016a, DanersGlueckKennedy2016b} and resulted in a variety of articles:
for instance, a characterization of uniform eventual positivity is given in \cite{DanersGlueck2018b}, perturbation results can be found in \cite{DanersGlueck2018a} (we also refer to \cite{ShakeriAlizadeh2017} for related finite-dimensional perturbation results), the long-term behaviour is studied in \cite{ArnoldCoine2023, Arora2023, AroraGlueck2021a, AroraGlueck2021c, Vogt2022}, and the more subtle property of \emph{local eventual positivity} is analyzed in \cite{Arora2022, Arora2023, DanersGlueckMui2023, Mui2023}.
An overview of the state of the art as of the beginning of 2022 can be found in \cite{Glueck2022}.

 In general, the theory incorporates two types of results in its current level of development:

\begin{enumerate}[label=\upshape(\roman*), leftmargin=*, widest=iii]
    \item 
    Characterizations of eventual positivity, though under rather strong technical assumptions.

    \item 
    Consequences of eventual positivity.
\end{enumerate}

The characterization results mentioned in~(i) have already been used to show eventual positivity (and versions thereof) for a large variety of concrete PDEs, for example, the bi-Laplacian with Wentzell boundary conditions \cite{DenkKunzePloss2021}, delay differential equations \cite{DanersGlueckKennedy2016a, DanersGlueck2018b, Glueck2016}, bi-Laplacians on graphs \cite{GregorioMugnolo2020a},
Laplacians with point interactions \cite{HusseinMugnolo2020}, and polyharmonic operators on graphs \cite{BeckerGregorioMugnolo2021}.

Results of type~(ii), on the other hand -- at the current state of the art -- are of a more theoretical nature: they give us necessary conditions for eventual positivity and indicate, at the same time, the limits of the theory.
This paper is mainly about the results of the second type.

\subsection*{Irreducibility for positive semigroups}

A $C_0$-semigroup on a Banach lattice $E$ is said to be \emph{positive} if each semigroup operator leaves the positive cone $E_+$ invariant and a positive $C_0$-semigroup is called \emph{irreducible} if it does not leave any closed ideals invariant except for $\{0\}$ and $E$ itself. 
This notion of irreducibility is motivated by the same concept in finite dimensions where it occurs, for instance, in the classical Perron--Frobenius theorem, in graph theory, and the theory of Markov processes on finite state spaces. 
In infinite dimensions, many positive semigroups that describe the solutions to concrete evolution equations are irreducible.
Detailed theoretical information about such semigroups can, for instance, be found in \cite[Section~C-III.3]{Nagel1986} and \cite[Section~14.3]{BatkaiKramarRhandi2017}. 
A treatment of irreducibility for general semigroups of positive operators -- i.e., not only one-parameter semigroups -- can be found in \cite{GaoTroitsky2014}. 

A property that makes irreducibility efficient to handle is that one can characterize it in terms of duality: a positive $C_0$-semigroup $(e^{tA})_{t \ge 0}$ on a Banach lattice $E$ is irreducible if and only if for each non-zero $f \in E_+$ and each non-zero positive functional $\varphi$ in the dual space $E'$ there exists a time $t \ge 0$ such that $\langle \varphi, e^{tA} f \rangle > 0$. 
This observation is classical, see for instance \cite[Definition~C-III.3.1]{Nagel1986}, but we include a rather detailed analysis of the situation in Proposition~\ref{prop:irreducibility-positive}.

\subsection*{Contributions and organization of the article}

Motivated by the recent theory and applications of eventually positive semigroups, we explore the notion of irreducibility outside the framework of positive semigroups and thereby focus on the eventually positive case. 
This gives rise to two problems that cannot be solved by falling back to classical arguments from the positive case and thus require the use of different methods: 

\begin{enumerate}[label=\upshape(\roman*), leftmargin=*, widest=iii]
    \item 
    Classically, the analysis of invariant ideals for semigroups strongly relies on the positivity of the semigroup operators.
    
    However, as the semigroups we are interested in are only eventually positive, it becomes natural to focus on (the non-existence of) ideals that are only eventually invariant. 
    For individually eventually positive semigroups though, this poses additional challenges since we cannot expect any of the operators in the semigroup to be positive. 
    
    \item 
    For positive semigroups, the above-mentioned fact that irreducibility can be characterized by means of duality is very powerful.
    However, the arguments used to show this characterization in the positive case cannot be directly adapted to treat the eventually positive situation.
\end{enumerate}

In the light of those challenges, we make the following contributions:

Regarding~(i), we distinguish between irreducible semigroups in the classical sense and what we call \emph{persistently irreducible} semigroups, that rely on eventual invariance of closed ideals (see Definition~\ref{def:irreducibility}). 
We show this distinction is indeed necessary in the eventually positive case (Example~\ref{exa:rademacher}), but not in the positive case (Proposition~\ref{prop:irreducibility-positive}). 
For eventually positive semigroups that are analytic, irreducibility and persistent irreducibility are also equivalent (Proposition~\ref{prop:irred-analytic}) and, in some cases, even imply a stronger version of eventual positivity (Proposition~\ref{prop:irred-analytic-strongly}); this is similar to the same phenomenon in the positive case \cite[Theorem~C-III-3.2(b) on p.\,306]{Nagel1986}.

Regarding~(ii), in order to characterize persistent irreducibility by means of duality (Theorem~\ref{thm:irreducibility-eventually-positive}) a new tool is required: we first develop conditions under which non-closed vector subspaces that are individually eventually invariant are automatically uniformly eventually invariant. 
This is possible in the setting of operator ranges that we discuss in Section~\ref{sec:eve-invarance}, see in particular Corollary~\ref{cor:ind-implies-unif-for-certain-operator-ranges}. 
To explain the relevance of this to our framework let us point out that, while we define persistent irreducibility by means of closed ideals, non-closed ideals still play an essential role in the theory. 
This is because important arguments rely on the construction of invariant principal ideals -- which are typically not closed (see the proof of Lemma~\ref{lem:sub-eigenvector}).

As mentioned before, positive irreducible semigroups exhibit many interesting spectral properties. Motivated by this, we study the spectrum of eventually positive persistently irreducible semigroups (Section~\ref{sec:spectrum}) and show that some of these properties are carried over from the positive case. In particular, we show that the spectrum of the generator of an eventually positive persistently irreducible semigroup is guaranteed to be non-empty in some situations.

The perturbation theory of eventually positive $C_0$-semigroups is currently quite limited. 
It was demonstrated in \cite{DanersGlueck2018a} that eventually positive semigroups do not behave well under (positive) perturbations. 
In Section~\ref{sec:perturbations} we show that if the perturbation interacts well with the unperturbed semigroup though, then the eventual positivity is indeed preserved. This enables us to construct eventually positive semigroups that do not satisfy the technical assumptions of the available characterization results in, for instance, \cite{DanersGlueckKennedy2016b} and \cite{DanersGlueck2018b}. 
Consequently, we are able to give an example of an eventually positive (but non-positive) semigroup that is persistently irreducible but not eventually strongly positive (Example~\ref{exa:coupling-irreducible-not-strong}). Analogously to the positive case, such a situation cannot occur for analytic semigroups (Proposition~\ref{prop:irred-analytic-strongly}).

\subsection*{Notation and terminology}

Throughout, we freely use the theory of Banach lattices for which we refer to the standard monographs \cite{Schaefer1974, Zaanen1983, Meyer-Nieberg1991, AliprantisBurkinshaw2006}. 
All Banach spaces and Banach lattices in the article are allowed to be real or complex unless otherwise specified.
Let $E$ be a Banach lattice with positive cone $E_+$. For $f\in E_+$, we alternatively use the notation $f\ge 0$ and write $f\gneq 0$ if $f\ge 0$ but $f\ne 0$.
For each $u\in E_+$, the \emph{principal ideal} of $E$ given by
\[
    E_u := \{f \in E: \text{ there exists }c>0\text { such that }\modulus{f}\le cu \}
\]
is also a Banach lattice when equipped with the \emph{gauge norm}
\[
    \norm{f}_u:= \inf\{c>0: \modulus{f}\le cu\}\qquad (f\in E_u).
\]
The principal ideal $E_u$ embeds continuously in $E$. In addition, if this embedding is dense, then $u$ is called a \emph{quasi-interior point} of $E_+$.

If $E$ is a complex Banach lattice, we denote its real part by $E_{\bbR}$. An operator $A: E \supseteq \dom{A} \to E$ is said to be \emph{real} if
\[
    \dom{A}=\dom{A} \cap E_{\bbR}+i\dom{A}\cap E_{\bbR} \quad \text{and} \quad A(\dom{A}\cap E_{\bbR}) \subseteq E_{\bbR}.
\]
In particular, a bounded operator $T$ between Banach lattices $E$ and $F$ is real if and only if $T(E_{\bbR})\subseteq F_{\bbR}$. We say that $T$ is positive if $T(E_+)\subseteq F_+$. In particular, every positive operator is real. Moreover, we say that $T$ is \emph{strongly positive} if $Tf$ is a quasi-interior point of $F_+$ for each $0\lneq f\in E$. We denote the space of bounded linear operators between $E$ and $F$ by $\calL(E,F)$ with the shorthand $\calL(E):=\calL(E,E)$. The range of $T\in \calL(E,F)$ is denoted by $\Ima T$.

A $C_0$-semigroup $(e^{tA})_{t\ge 0}$ is said to be \emph{real} if each operator $e^{tA}$ is real. It can be easily seen that $(e^{tA})_{t\ge 0}$ is real if and only if $A$ is a real operator. Lastly, we recall that the \emph{spectral bound} $\spb(A)$ of the generator $A$ is defined as
\[
    \spb(A):=\sup\{\re \lambda: \lambda\in \spec(A)\} \in [-\infty,\infty).
\]

\section{Eventual invariance of vector subspaces}
    \label{sec:eve-invarance}

We introduce the notion of a persistently irreducible $C_0$-semigroup in Section~\ref{sec:irreducibility} using the concept of eventually invariant ideals. To aid us in the sequel, we collect here a few preliminary results about eventual invariance.

\begin{definition}
        Let $E$ be a Banach space and let $\calT = (T_i)_{i \in I}$ a net of bounded linear operators on $E$.
	A set $S \subseteq E$ is called\dots 
	\begin{enumerate}[label=\upshape(\alph*), leftmargin=*, widest=iii]
		\item 
		\dots \emph{invariant} under $\calT$ if $T_i S \subseteq S$ for each $i \in I$.
		
		\item 
		\dots \emph{uniformly eventually invariant} under $\calT$ if there exists $i_0 \in I$ such that $T_i S \subseteq S$ for each $i \ge i_0$.
		
		\item 
		\dots \emph{individually eventually invariant} under $\calT$ if for each $f \in S$ there exists $i_0 \in I$ such that $T_i f \in S$ for all $i \ge i_0$. 
	\end{enumerate}
\end{definition}

In the subsequent sections, we will use these concepts specifically for $C_0$-semigroups to study irreducibility and versions thereof.
To this end it will be important to understand under which conditions an individually eventually invariant vector subspace is even uniformly eventually invariant. 
For the case of closed vector subspaces, this is quite easy as long as the index set of the net of operators is not too large. Recall that a subset $B$ of a directed set $(A,\le)$ is said to be \emph{majorizing} if for each $a \in A$, there exists $b\in B$ such that $a\le b$. 
We first consider nets of operators which eventually map into a fixed closed vector subspace; 
as a corollary, we then obtain a result for the eventual invariance of closed vector subspaces.

\begin{proposition}
    \label{prop:ind-implies-unif-for-closed-subspaces-absorbing}
    Let $E,F$ be Banach spaces and let $\calT = (T_i)_{i \in I}$ a net of bounded linear operators from $E$ to $F$ 
    such that the directed set $I$ contains a countable majorizing subset.
    Let $W \subseteq F$ be a closed vector subspace. 

    Assume that for every $x \in E$ there exists $i_0 \in I$ such that $T_i x \in W$ for all $i \ge i_0$. 
    Then $i_0$ can be chosen to be independent of $x$. 
\end{proposition}

\begin{proof}
    Let $J \subseteq I$ be a countable majorizing subset. 
    For each $j \in J$ the set 
    \begin{align*}
        V_j 
        :=
        \{x \in E : \, T_ix \in W \text{ for all } i \ge j\}
    \end{align*}
    is a closed subspace of $E$ and it follows from the assumption that $\bigcup_{j \in J} V_j = E$. 
    Hence, by Baire's theorem, there exists $j_0 \in J$ such that $V_{j_0} = E$. 
\end{proof}

\begin{corollary}
    \label{cor:ind-implies-unif-for-closed-subspaces}
    Let $E$ be a Banach space and let $\calT = (T_i)_{i \in I}$ be a net of bounded linear operators on $E$ 
    such that the directed set $I$ contains a countable majorizing subset.
    Let $V \subseteq E$ be a closed vector subspace. 
	
    If $V$ is individually eventually invariant under $\calT$, then it is even uniformly eventually invariant under $\calT$.
\end{corollary}

\begin{proof}
    Consider the restricted operators $T_i\restrict{V}: V \to E$ and apply Proposition~\ref{prop:ind-implies-unif-for-closed-subspaces-absorbing}.
\end{proof}

We note that if a set $S \subseteq E$ is uniformly eventually invariant, then so is its closure $\overline{S}$. 
However, individual eventual invariance of $S$ does not imply individual eventual invariance of $\overline{S}$, in general.
This is a serious obstacle in the subsequent section, especially in the proof of Theorem~\ref{thm:irreducibility-eventually-positive}. 
To overcome it, we need a generalisation of Corollary~\ref{cor:ind-implies-unif-for-closed-subspaces} to the so-called \emph{operator ranges}. 

A vector subspace $V$ of a Banach space $E$ is called an \emph{operator range} in $E$ if there exists a complete norm $\norm{\argument}_V$ on $V$ which makes the embedding of $V$ into $E$ continuous. 
If existent, such a norm on $V$ is uniquely determined up to equivalence due to the closed graph theorem.
It follows easily from a quotient space argument that $V$ is an operator range if and only if there exists a Banach space $D$ and a bounded linear operator $T: D \to E$ with range $V$ (this explains the terminology \emph{operator range}). The concept of operator ranges was studied in depth in \cite{Cross1980, CrossOstrovskiiShevchik1995} and was recently used to prove an abstract characterization of the individual maximum and anti-maximum principle in \cite{AroraGlueck2022a} (see also \cite[Chapter~4]{Arora2023}).

The intersection of finitely many operator ranges is again an operator range (the proof is easy; details can be found in \cite[Proposition~2.2]{Cross1980} or \cite[Proposition~4.1.5]{Arora2023}). 
The next example shows that the intersection of infinitely many operator ranges is not an operator range, in general. 
The subsequent Proposition~\ref{prop:intersection-of-operator-ranges}, however, serves as a reasonable substitute.

\begin{example}
	Let $Q$ denote the set of all quasi-interior points in the cone of the Banach lattice $c_0$, i.e., the set of all $0 \le f \in c_0$ such that $f_k > 0$ for each $k \in \bbN$. 
	For each $f \in Q$ the principal ideal $(c_0)_f$ is an operator range, being complete with respect to the gauge norm $\norm{\argument}_f$. 
	However, the intersection 
	\begin{align*}
		\bigcap_{f \in Q} (c_0)_f
	\end{align*}
	can easily be checked to coincide with the space $c_{00}$ of sequences with only finitely many non-zero entries, and this space cannot be an operator range since it has a countable Hamel basis (so in particular, it is not complete under any norm).
\end{example}

\begin{proposition}
    \label{prop:intersection-of-operator-ranges}
    Let $E$ be a Banach space, let $I$ be a non-empty set, and for each $i \in I$, let $V_i \subseteq E$ be an operator range with a complete norm $\norm{\argument}_{i}$ which makes the embedding of $V_i$ into $E$ continuous.
    Moreover, for each $i \in I$, let $c_i > 0$ be a real number.
    Then the vector subspace
    \begin{align*}
        V := \Big\{ v \in \bigcap_{i \in I} V_i : \; \norm{v}_V := \sup_{i \in I} c_i \norm{v}_{i} < \infty \Big\}
    \end{align*}
    is complete when equipped with the norm $\norm{\argument}_V$ and the corresponding embedding into $E$ is continuous. In particular, $V$ is an operator range.
\end{proposition}

It is important to note that the space $(V, \norm{\argument}_V)$ does not only depend on the numbers $c_i$, 
but also on the choice of the norms $\norm{\argument}_{i}$ 
-- while for fixed $i$ all norms on $V_i$ that make the embedding $V_i \to E$ continuous are equivalent, we may re-scale each norm with an $i$-dependent factor. 
For this reason, we can assume in the proof of the proposition that all factors $c_i$ are equal to $1$;
this can be achieved by simply replacing each norm $\norm{\argument}_{i}$ with the norm $c_i \norm{\argument}_{i}$.

\begin{proof}[Proof of Proposition~\ref{prop:intersection-of-operator-ranges}]
    As mentioned before the proof we assume that $c_i = 1$ for each $i$.
    Clearly, the normed space $(V, \norm{\argument}_V)$ embeds continuously into $(V_i, \norm{\argument}_{i})$ for each $i$, so it also embeds continuously into $E$. 
	
    It suffices now to show that $(V, \norm{\argument}_V)$ is complete. 
    To this end, let $(x^n)$ be a Cauchy sequence in $(V,\norm{\argument}_V)$. 
    Then $(x^n)$ is bounded in $(V,\norm{\argument}_V)$, say by a number $M \ge 0$.
    For each $i$, as $(V, \norm{\argument}_V)$ embeds continuously into the Banach space $(V_i, \norm{\argument}_{i})$, there exists $x_i\in V_i$ such that $x^n \to x_i$ in $(V_i, \norm{\argument}_{i})$. Now by the continuous embedding of $V_i$ into $E$, we  obtain that $x^n \to x_i$ in $E$ for each $i$. 
    It follows that $x:=x_i=x_j$ for all $i,j\in I$. 
    In particular, $x \in \bigcap_{i \in I} V_i$.
    
    To show show that $x \in V$ and that $(x^n)$ converges to $x$ with respect to $\norm{\argument}_{V}$, let $\varepsilon > 0$. 
    Since $(x^n)$ is a Cauchy sequence in $(V, \norm{\argument}_{V})$, 
    there exists $n_0$ such that $\norm{x^m-x^n}_{i} \le \varepsilon$ for all $i \in I$ and all $n,m \ge n_0$. 
    For each $i \in I$, the convergence of $x^n$ to $x$ with respect to $\norm{\argument}_{i}$ thus yields $\norm{x^m - x}_{i} \le \varepsilon$ for all $m \ge n_0$. 
    Hence, $\norm{x}_i \le \norm{x^{n_0}}_i + \varepsilon \le M+\varepsilon$, for each $i \in I$, so $x \in V$.
    Moreover, the inequality $\norm{x^m - x}_{i} \le \varepsilon$ for each $i \in I$ and all $m \ge n_0$ shows that $\norm{x^m-x}_V \le \varepsilon$ for all $m \ge n_0$, so indeed $x^n \to x$ in $(V, \norm{\argument}_V)$.
\end{proof}

We are now in a position to show the following generalization of Proposition~\ref{prop:ind-implies-unif-for-closed-subspaces-absorbing}. 
As a consequence, we will obtain a generalization of the eventual invariance result in Corollary~\ref{cor:ind-implies-unif-for-closed-subspaces} to operator ranges (Corollary~\ref{cor:ind-implies-unif-for-certain-operator-ranges}).

\begin{theorem}
    \label{thm:ind-implies-unif-for-certain-operator-ranges-absorbing}
    Let $E,F$ be Banach spaces and let $\calT = (T_i)_{i \in I}$ a net of bounded linear operators from $E$ to $F$ 
    such that the directed set $I$ contains a countable majorizing subset.
    Let $W \subseteq E$ be an operator range, endowed with a complete norm $\norm{\argument}_W$ that makes the inclusion map $W \to E$ continuous. 
    The following assertions are equivalent:
    \begin{enumerate}[label=\upshape(\roman*), leftmargin=*, widest=iii]
        \item\label{thm:ind-implies-unif-for-certain-operator-ranges:itm:uniform} 
        There exists $i_0 \in I$ such that $T_i E \subseteq W$ for each $i \ge i_0$.
        
        \item\label{thm:ind-implies-unif-for-certain-operator-ranges:itm:individual} 
        There are numbers $c_i \geq 0$ for all $i \in I$ that satisfy the following property:
        For each $x \in E$ there exists $i_0 \in I$ such that $T_i x \in W$ and $\norm{T_i x}_W \le c_i \norm{x}_E$ for all $i \ge i_0$.
    \end{enumerate}
\end{theorem}

A result that is loosely reminiscent of this theorem can be found in \cite[Theorem~2.1]{AroraGlueck2022a}. 
For the proof of Theorem~\ref{thm:ind-implies-unif-for-certain-operator-ranges-absorbing} we need the following observation~\cite[Proposition~2.5]{AroraGlueck2022a}: 
if $W$ is an operator range in a Banach space $F$, and $T: E \to F$ is a bounded linear operator from another Banach space $E$ into $F$, then the pre-image $V := T^{-1}(W)$ is an operator range in $E$, as the norm $\norm{\argument}_{V}$ given by
\begin{align*}
    \norm{x}_{V} := \norm{x}_E + \norm{Tx}_W 
\end{align*}
for all $x \in V$ is complete and makes the embedding of $V$ into $E$ continuous.

\begin{proof}[Proof of Theorem~\ref{thm:ind-implies-unif-for-certain-operator-ranges-absorbing}]
    Let $W\subseteq E$ be an operator range, endowed with a complete norm $\norm{\argument}_W$ that makes the embedding into $E$ continuous.
    
    ``(i) $\Rightarrow$ (ii)'': This implication immediately follows by choosing $c_i:= \norm{T_i}_{E\to W}$ -- which is finite by closed graph theorem -- for $i\ge i_0$ and $c_i:=0$ for all other $i$.

    ``(ii) $\Rightarrow$ (i)'': Without loss of generality, we assume that $c_i \ge 1$ for each $i\in I$. Let $J\subseteq I$ be a countable majorizing set and for each $j\in J$, let
    \[
        V_j := \bigcap_{i\ge j} T_{i}^{-1}(W).
    \]
    Since $W$ is an operator range in $F$, for each $i \in I$ the subspace $T_{i}^{-1}(W)$ is an operator range in $E$ as it is complete 
    and embeds continuously into $E$ when endowed with the norm 
    given by $\norm{x}_i:= \norm{x}_E+\norm{T_ix}_W$ for all $x \in T_{i}^{-1}(W)$. 
    Thus, by Proposition~\ref{prop:intersection-of-operator-ranges}, the subspace
    \[
        \tilde{V_j}:= \{x\in V_j : \norm{x}_{\tilde{V_j}} := \sup_{i\ge j} c_i^{-1} \norm{x}_i <\infty\}
    \]
    of $E$ is also an operator range for each $j \in J$.

    Next, we observe that $\bigcup_{j\in J} \tilde{V_j} = E$. 
    Indeed, let $x\in E$. As $J$ is majorizing we can, due to assumption~(ii), choose $j\in J$ such that for each $i\ge j$, we have $T_ix\in W$ and $\norm{T_i x}_W \le c_i \norm{x}_E$.
    Hence, for all $i \ge j$, one has $\norm{x}_i \le (1+c_i) \norm{x}_E$ and thus 
    \begin{align*}
        c_i^{-1} \norm{x}_i \le 2 \norm{x}_E,
    \end{align*}
    as we chose each $c_i$ to be at least $1$.
    It follows that $x\in \tilde{V_j}$, as desired. 

    Since $\bigcup_{j\in J} \tilde{V_j} = E$, we can apply Baire's theorem for operator ranges \cite[Proposition~2.6]{AroraGlueck2022a} 
    to conclude that there exists $j\in J$ such that $\tilde{V_j}=E$. 
    In turn, $V_j = E$ and thus $T_i E \subseteq W$ for all $i\ge j$.
\end{proof}

\begin{corollary}
    \label{cor:ind-implies-unif-for-certain-operator-ranges}
    Let $E$ be a Banach space and let $\calT = (T_i)_{i \in I}$ a net of bounded linear operators on $E$ 
    such that the directed set $I$ contains a countable majorizing set.
    For each operator range $V$ in $E$ the following assertions are equivalent:
    \begin{enumerate}[label=\upshape(\roman*), leftmargin=*, widest=iii]
        \item\label{cor:ind-implies-unif-for-certain-operator-ranges:itm:uniform} 
        The space $V$ is uniformly eventually invariant under $T$.
        
        \item\label{cor:ind-implies-unif-for-certain-operator-ranges:itm:individual} 
        There are numbers $c_i \geq 0$ for all $i \in I$ such that the space $V$ satisfies the following quantified individual eventual invariance property:
        
        For each $v \in V$ there exists $i_0 \in I$ such that $T_i v \in V$ and $\norm{T_i v}_V \le c_i \norm{v}_V$ for all $i \ge i_0$.
    \end{enumerate}
\end{corollary}

\begin{proof}
    Apply Theorem~\ref{thm:ind-implies-unif-for-certain-operator-ranges-absorbing} to the restricted operators $T_i\restrict{V} : V \to E$.
\end{proof}

In the context of Theorem~\ref{thm:ind-implies-unif-for-certain-operator-ranges-absorbing} and Corollary~\ref{cor:ind-implies-unif-for-certain-operator-ranges} it is worthwhile to mention the general question under which conditions individual properties of operator-valued functions are automatically uniform. 
An abstract framework to study this question was developed by Peruzzetto in~\cite{Peruzzetto2022}.

\section{Irreducibility of $C_0$-semigroups}
    \label{sec:irreducibility}

In this section, we study the irreducibility of $C_0$-semigroups, a notion which stems from the theory of positive semigroups. 
We drop the positivity assumption and first discuss a variety of sufficient conditions for the irreducibility of general $C_0$-semigroups (Proposition~\ref{prop:irreducibility-general}); those conditions are necessary in the positive case (Proposition~\ref{prop:irreducibility-positive}). 
In the eventually positive case, things turn out to be more involved (Example~\ref{exa:rademacher}) and the stronger notion of \emph{persistent irreducibility} turns out to be fruitful for the analysis, see in particular Theorem~\ref{thm:irreducibility-eventually-positive}.

Note that, as $[0,\infty)$ contains the majorizing set $\bbN$, individual and uniform eventual invariance of a closed subspace under a semigroup are equivalent notions (Corollary~\ref{cor:ind-implies-unif-for-closed-subspaces}).

\begin{definition}
    \label{def:irreducibility}
	A semigroup $(e^{tA})_{t \ge 0}$ on a Banach lattice $E$ is called\dots
	\begin{enumerate}[label=\upshape(\alph*), leftmargin=*, widest=iii]
		\item 
		\dots \emph{irreducible} if it has no closed invariant ideals, except for $\{0\}$ and $E$.
		
		\item 
		\dots \emph{persistently irreducible} if it has no closed ideal that is uniformly (equivalently: individually) eventually invariant, except for $\{0\}$ and $E$.
	\end{enumerate}
\end{definition}

Observe that both irreducibility and persistent irreducibility do not change if we replace $A$ with $A + \lambda$ for any $\lambda \in \bbC$.

The name \emph{persistently irreducible} is motivated by the observation that this property means that all the semigroup tails $(e^{tA})_{t \ge t_0}$ for $t_0 \ge 0$ act irreducibly on $E$, i.e., the semigroup is not only irreducible but it also remains irreducible when its action is only considered for large times. 
Obviously, persistent irreducibility implies irreducibility
(which also explains why we avoid the terminology \emph{eventually irreducible}, a notion that one might, at first glance, be tempted to use instead of \emph{persistently irreducible}).

From the definition, it is seen at once that nilpotent semigroups are not persistently irreducible unless $\dim E\le1$. In Example~\ref{exa:rademacher}, we give an example of a semigroup that is nilpotent and irreducible -- thus, showing that the notions irreducibility and persistent irreducibility are not equivalent. 
Throughout we will study a number of sufficient or necessary conditions for irreducibility; they are motivated by \cite[Definition~C-III.3.1(ii)]{Nagel1986}, where positive semigroups are studied.

 \begin{conditions}
    \label{con:irreducibility}
    Let $(e^{tA})_{t\ge 0}$ be a $C_0$-semigroup on a Banach lattice $E$. We study the relationships between the following assertions.
    \begin{enumerate}[label=\upshape(\roman*), leftmargin=*, widest=iii]
	\item\label{con:irreducibility:itm:irred} 
	The semigroup $(e^{tA})_{t \ge 0}$ is irreducible.
		
	\item\label{con:irreducibility:itm:weak-arbi-t} 
	\emph{Weak condition at arbitrary times:} 

        \noindent
        For each $0 \lneq f \in E$ and each $0 \lneq \varphi \in E'$, there exists $t \in [0,\infty)$ such that $\langle \varphi, e^{tA}f \rangle \not= 0$.
		
	\item\label{con:irreducibility:itm:pers-irred} 
	The semigroup $(e^{tA})_{t \ge 0}$ is persistently irreducible.
		
	\item\label{con:irreducibility:itm:weak-large-t-or-0} 
        \emph{Weak condition at large times or $0$:} 

        \noindent
	For each $0 \lneq f \in E$, each $0 \lneq \varphi \in E'$ and each $t_0 \in [0,\infty)$, there exists $t \in \{0\}\cup [t_0,\infty)$ such that $\langle \varphi, e^{tA}f \rangle \ne 0$.
		
	\item\label{con:irreducibility:itm:weak-large-t}
        \emph{Weak condition at large times:} 

        \noindent
	For each $0 \lneq f \in E$, each $0 \lneq \varphi \in E'$ and each $t_0 \in [0,\infty)$, there exists $t \in [t_0,\infty)$ such that $\langle \varphi, e^{tA}f \rangle \ne 0$.
    \end{enumerate}
\end{conditions}

Let us first point out that several implications between these conditions are true for general $C_0$-semigroups:

\begin{proposition}	
    \label{prop:irreducibility-general}
    For a $C_0$-semigroup $(e^{tA})_{t\ge 0}$ on a Banach lattice $E$, the following implications between the Conditions~\ref{con:irreducibility} are true:
    
    \centerline{
	\xymatrix{
            \parbox{3cm}{
                \centering 
                \ref{con:irreducibility:itm:pers-irred}\\
                \emph{persistent irreducibility}
            }
            \ar@{=>}[d]
		&
            \parbox{3.5cm}{
                \centering 
                \ref{con:irreducibility:itm:weak-large-t-or-0}\\
                \emph{weak condition at large times or $0$}
            }
            \ar@{=>}[l]
            \ar@{=>}[d]
            &
            \parbox{3cm}{
                \centering 
                \ref{con:irreducibility:itm:weak-large-t}\\
                \emph{weak condition at large times}
            }
            \ar@{=>}[l]
		\\
            \parbox{3cm}{
                \centering 
                \ref{con:irreducibility:itm:irred}\\
                \emph{irreducibility}
            }
		&
            \parbox{3.5cm}{
                \centering 
                \ref{con:irreducibility:itm:weak-arbi-t}\\
                \emph{weak condition at arbitrary times}
            }
            \ar@{=>}[l]
        &
	}
    }
\end{proposition}

\begin{proof}
    The implications
    ``\ref{con:irreducibility:itm:pers-irred} $\Rightarrow$ \ref{con:irreducibility:itm:irred}''
    and 
    ``\ref{con:irreducibility:itm:weak-large-t} $\Rightarrow$ \ref{con:irreducibility:itm:weak-large-t-or-0}
    $\Rightarrow$ \ref{con:irreducibility:itm:weak-arbi-t}''
    are obvious.

    ``\ref{con:irreducibility:itm:weak-arbi-t} $\Rightarrow$ \ref{con:irreducibility:itm:irred}'': 
    If~\ref{con:irreducibility:itm:irred} fails, we can find a closed ideal $I$ that is neither equal to $\{0\}$ nor to $E$, but that is invariant under the semigroup. 
    As $I$ is proper and non-zero, there exists a vector $0 \lneq f \in I$ and a functional $0 \lneq \varphi \in E'$ that vanishes on $I$. 
    One has $\langle \varphi, e^{tA} f \rangle = 0$ for all $t \ge 0$, so~\ref{con:irreducibility:itm:weak-arbi-t} fails.
	
    ``\ref{con:irreducibility:itm:weak-large-t-or-0} $\Rightarrow$ \ref{con:irreducibility:itm:pers-irred}'':
    If~\ref{con:irreducibility:itm:pers-irred} fails, we can find a closed ideal $I$ that is neither equal to $\{0\}$ nor to $E$ and a time $t_0 \ge 0$ such that $e^{tA}I \subseteq I$ for all $t \in [t_0,\infty)$.
    Again, as $I$ is proper and non-zero, there is $0 \lneq f \in I$ and a functional $0 \lneq \varphi \in E'$ which vanishes on $I$. 
    Hence, we have $\langle \varphi, e^{tA} f \rangle = 0$ for all $t \in \{0\}\cup [t_0,\infty)$, which shows that~\ref{con:irreducibility:itm:weak-large-t-or-0} fails.
\end{proof}

\begin{remark}
    Proposition~\ref{prop:irreducibility-general} gives us a plethora of examples of (non-positive) semigroups that are (persistently) irreducible. Indeed, we see that persistent irreducibility is implied by \emph{individual eventual strong positivity}, i.e., by the property
    \begin{equation}
        \label{eq:eventual-strong-positivity}
        \forall f\gneq 0\ \exists t_0\ge 0 \ \forall t\ge t_0: e^{tA}f \text{ is a quasi-interior point of } E_+; 
    \end{equation}
    this follows from the fact that a vector $g \in E_+$ is a quasi-interior point of $E_+$ if and only if $\langle \varphi, g \rangle > 0$ for all $0 \lneq \varphi \in E'$, see \cite[Theorem~II.6.3]{Schaefer1974}. Therefore, all examples of eventually strongly positive semigroups in \cite[Section~6]{DanersGlueckKennedy2016b} and \cite[Section~4]{DanersGlueck2018b} are persistently irreducible. In Example~\ref{exa:coupling-irreducible-not-strong}, we give an example of a non-positive persistently irreducible semigroup that does not satisfy~\eqref{eq:eventual-strong-positivity}.
\end{remark}

For positive semigroups, the notions of irreducibility and persistent irreducibility are, in fact, equivalent:

\begin{proposition}
    \label{prop:irreducibility-positive}
    For a positive $C_0$-semigroup $(e^{tA})_{t\ge 0}$ on a Banach lattice $E$, all five of the Conditions~\ref{con:irreducibility} are equivalent.
\end{proposition}

\begin{proof}
    By Proposition~\ref{prop:irreducibility-general} it suffices to prove the following two implications:
    
    ``\ref{con:irreducibility:itm:irred} $\Rightarrow$ \ref{con:irreducibility:itm:weak-arbi-t}'': 
    This implication is well-known for positive semigroups. 
    The argument can be found just after \cite[Definition~B-III-3.1]{Nagel1986} (see also \cite[Section~C-III.3]{Nagel1986}).
 
    ``\ref{con:irreducibility:itm:weak-arbi-t} $\Rightarrow$ \ref{con:irreducibility:itm:weak-large-t}'':
    Since~\ref{con:irreducibility:itm:weak-arbi-t} holds, we know from Proposition~\ref{prop:irreducibility-general} that the semigroup is also irreducible. 
    The irreducibility together with the positivity implies that each operator $e^{tA}$ is \emph{strictly positive}, meaning that $e^{tA}f \gneq 0$ whenever $f \gneq 0$; see \cite[Theorem~C-III.3.2(a)]{Nagel1986}. 
    So if $f$, $\varphi$, and $t_0$ are given as in~\ref{con:irreducibility:itm:weak-large-t}, then $e^{t_0 A}f \gneq 0$. 
	By applying~\ref{con:irreducibility:itm:weak-arbi-t} to the vectors $e^{t_0 A}f$ and $\varphi$ we find a time $t \ge 0$ such that $\langle \varphi, e^{(t + t_0)A}f \rangle \ne 0$.
\end{proof}

The situation gets more subtle for eventually positive semigroups. 
For them, the following theorem indicates that persistent irreducibility is the appropriate notion to further build the theory on, as this property can be conveniently characterized by testing against functionals.

\begin{theorem}
    \label{thm:irreducibility-eventually-positive}
    For an individually eventually positive real $C_0$-semigroup $(e^{tA})_{t\ge 0}$ on a Banach lattice $E$, the following implications between the Conditions~\ref{con:irreducibility} are true:
    
    
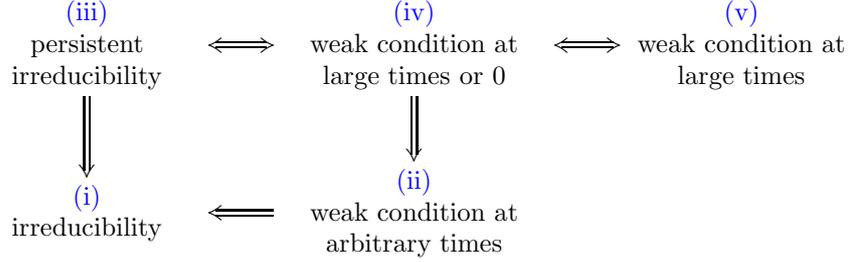
\centerline{
	\xymatrix{
            \parbox{3cm}{
                \centering 
                \ref{con:irreducibility:itm:pers-irred}\\
                \emph{persistent irreducibility}
            }
            \ar@{=>}[d]
		&
            \parbox{3.5cm}{
                \centering 
                \ref{con:irreducibility:itm:weak-large-t-or-0}\\
                \emph{weak condition at large times or $0$}
            }
            \ar@{<=>}[l]
            \ar@{=>}[d]
            &
            \parbox{3cm}{
                \centering 
                \ref{con:irreducibility:itm:weak-large-t}\\
                \emph{weak condition at large times}
            }
            \ar@{<=>}[l]
		\\
            \parbox{3cm}{
                \centering 
                \ref{con:irreducibility:itm:irred}\\
                \emph{irreducibility}
            }
		&
            \parbox{3.5cm}{
                \centering 
                \ref{con:irreducibility:itm:weak-arbi-t}\\
                \emph{weak condition at arbitrary times}
            }
            \ar@{=>}[l]
        &
	}
    }
\end{theorem}

The difference to the situation without any eventual positivity assumption (Proposition~\ref{prop:irreducibility-general}) is that one now has equivalences throughout the first row of the diagram.
In light of the situation for positive semigroups (Proposition~\ref{prop:irreducibility-positive}) one may ask whether Theorem~\ref{thm:irreducibility-eventually-positive} can be improved to get an equivalence between all five conditions, say at least for uniformly eventually positive semigroups. A negative answer is given in Example~\ref{exa:rademacher}. However, the situation improves significantly if the semigroup is, in addition, analytic (Proposition~\ref{prop:irred-analytic}).

For the proof of Theorem~\ref{thm:irreducibility-eventually-positive}, we need the following sufficient condition for a principal ideal in a Banach lattice to be uniformly eventually invariant under a semigroup. 
Getting uniform (rather than only individual) eventual invariance in the lemma is a bit subtle, as we merely assume the semigroup to be individually eventually positive.
This is where our preparations on eventually invariant operator ranges from Section~\ref{sec:eve-invarance} enter the game.

\begin{lemma}
    \label{lem:sub-eigenvector}
    Let $E$ be a Banach lattice and let $(e^{tA})_{t \ge 0}$ be an individually eventually positive real $C_0$-semigroup on $E$.
	
    Let $0 \le h \in E$ and assume that there exists a time $t_0 \ge 0$ such that $e^{tA} h \le h$ for all $t \ge t_0$.
    Then the principal ideal $E_h$ (and, in turn, its closure) is uniformly eventually invariant under $(e^{tA})_{t \ge 0}$.
\end{lemma}

\begin{proof}
    We first consider a real vector $f$ in the order interval $[-h, h]$. 
    Due to the individual eventual positivity of the semigroup there exists a time $t_1 \ge t_0$ 
    such that $e^{tA}(h-f) \ge 0$ and $e^{tA}(h+f) \ge 0$ for all $t \ge t_1$. 
    Thus, for all $t \ge t_1 \ge t_0$, the vector $e^{tA}f$ is real and satisfies
    \begin{align*}
	\pm e^{tA}f \le e^{tA}h \le h,
    \end{align*}
    so $e^{tA}f \in [-h,h]$. 
    This proves that the order interval $[-h,h]$ is individually eventually invariant under the semigroup. 
    As $[-h,h]$ spans $E_h$ (over the underlying scalar field), 
    it follows that $E_h$ is individually eventually invariant under the semigroup. 
    
    To show that $E_h$ is even uniformly eventually invariant, we will now employ Corollary~\ref{cor:ind-implies-unif-for-certain-operator-ranges}.	
    If the underlying scalar field is $\bbR$, the preceding argument shows that, for each $f \in E_h$, 
    we have $\norm{e^{tA} f}_h \le \norm{f}_h$ for all sufficiently large times $t$. 
    If the underlying scalar field is $\bbC$ and $f \in E_h$, say with gauge norm $\norm{f}_h \le 1$, 
    we can write $f$ as $f = f_1 + i f_2$ for real vectors $f_1,f_2 \in [-h,h]$ and hence, 
    $e^{tA}f = e^{tA}f_1 + i e^{tA}f_2$ has modulus at most $2h$ for all sufficiently large $t$.
    This shows that, for each $f \in E_h$, one has  $\norm{e^{tA}f}_h \le 2\norm{f}_h$ for all sufficiently large times $t$. 
    In both cases, Corollary~\ref{cor:ind-implies-unif-for-certain-operator-ranges} can be applied and gives the uniform eventual invariance of $E_h$.
\end{proof}

\begin{proof}[Proof of Theorem~\ref{thm:irreducibility-eventually-positive}]
    Due to Proposition~\ref{prop:irreducibility-general}, only one implication is left to prove, namely
    ``\ref{con:irreducibility:itm:pers-irred} $\Rightarrow$ \ref{con:irreducibility:itm:weak-large-t}''.
    Without loss of generality, assume that the growth bound $\gbd(A)$ of the semigroup satisfies $\gbd(A) < 0$.
	
    Assume that~\ref{con:irreducibility:itm:weak-large-t} fails, i.e., we can find $0 \lneq f \in E$, $0 \lneq \varphi \in E'$ and $t_0 \in [0,\infty)$ such that $\langle \varphi, e^{tA}f \rangle = 0$ for all $t \in  [t_0,\infty)$.  
    Due to the individual eventual positivity of the semigroup, we can find a time $t_1 \ge t_0$ such that $e^{tA}f \ge 0$ for each $t \ge t_1$. 
    We distinguish two cases:
	
    \emph{Case 1:} $e^{t_1A}f = 0$.
    It then follows from the individual eventual positivity that the orbit of every $g \in E_f$ under the semigroup eventually vanishes. 
    By applying Proposition~\ref{prop:ind-implies-unif-for-closed-subspaces-absorbing} to the family of operators $\big(e^{tA}\restrict{E_f}\big)_{t\ge 0}$ and to $W=\{0\}$, we conclude that, for all sufficiently large $t$,
    $e^{tA}$ vanishes on $E_f$ and hence on $\overline{E_f}$.
    
    In particular, the closed ideal $\overline{E_f}$ is uniformly eventually invariant. Thus, if $\overline{E_f} \not= E$, then~\ref{con:irreducibility:itm:pers-irred} fails (as $f$ is non-zero). 
    
    On the other hand, if $\overline{E_f} = E$, then the semigroup is nilpotent. 
    But since~\ref{con:irreducibility:itm:weak-large-t} is not true, one has $\dim E>1$. 
    Together with the nilpotency, this shows that the semigroup cannot be persistently irreducible, 
    i.e.,~\ref{con:irreducibility:itm:pers-irred} fails.
	
    \emph{Case 2:} $e^{t_1A}f \not= 0$.
    By the continuity of the semigroup orbit of $f$ and the inequality $e^{tA}f \ge 0$ for $t \ge t_1$, it follows (by testing against positive functionals) that
    \begin{align*}
	h := \int_{t_1}^\infty e^{tA} f \dx t \gneq 0;
    \end{align*}
    the convergence of the integral is guaranteed as we assumed $\omega_0(A) < 0$. 
    By using again that $e^{tA}f \ge 0$ for all $t \ge t_1$, one readily checks that $e^{tA}h \le h$ for all $t \ge 0$.
    So, according to Lemma~\ref{lem:sub-eigenvector}, the closure of the ideal $E_h$ is uniformly eventually invariant under the semigroup. 
    This closed ideal is non-zero since it contains that non-zero vector $h$. 
    Moreover, we have $\langle \varphi, h \rangle = 0$, so $\varphi$ vanishes on $\overline{E_h}$, which shows that $\overline{E_h} \not= E$.
    Whence, the semigroup is not persistently irreducible, i.e.,~\ref{con:irreducibility:itm:pers-irred} fails.
\end{proof}

As a consequence of Theorem~\ref{thm:irreducibility-eventually-positive}, we obtain the following eventual \emph{strict positivity} result for eventually positive persistently irreducible semigroups:

\begin{corollary}
    \label{cor:persistently-irred-non-zero}
    Let $E$ be a Banach lattice and assume that $(e^{tA})_{t \ge 0}$ is a real $C_0$-semigroup on $E$ that is individually eventually positive and persistently irreducible.

    For all $0 \lneq f \in E$ and $0 \lneq \varphi \in E'$ and all times $t \ge 0$ one has $e^{tA}f \not= 0$ and $(e^{tA})' \varphi \not= 0$.
\end{corollary}

\begin{proof}
    Due to Theorem~\ref{thm:irreducibility-eventually-positive} the semigroup satisfies Condition~\ref{con:irreducibility}\ref{con:irreducibility:itm:weak-large-t}.

    Let $0 \lneq f \in E$ and $0 \lneq \varphi \in E'$. Suppose there exists $t_0>0$ such that $e^{t_0 A}f=0$ or $(e^{t_0A})' \varphi = 0$. Then $e^{tA}f=0$ or $(e^{tA})' \varphi = 0$ for all $t\ge t_0$. In either case, $\duality{\varphi}{e^{tA}f}=0$ for all $t\ge t_0$, which contradicts Condition~\ref{con:irreducibility}\ref{con:irreducibility:itm:weak-large-t}.
\end{proof}

\begin{remark}
    Using Corollary~\ref{cor:persistently-irred-non-zero}, we are able to obtain the following analogue of \cite[Theorem~C-III.3.2(a)]{Nagel1986}: 
    
    If $(e^{tA})_{t \ge 0}$ is a $C_0$-semigroup on $E$ that is uniformly eventually positive and persistently irreducible, then there exists a time $t_0\ge 0$ such that $e^{tA}$ is a strictly positive operator for all $t\ge t_0$, meaning that $e^{tA}f\gneq 0$ for all $0\lneq f\in E$ and all $t\ge t_0$.
\end{remark}

There exist uniformly eventually positive $C_0$-semigroups 
which are irreducible but not persistently irreducible; 
here is a concrete example.

\begin{example}
    \label{exa:rademacher}
    \emph{A uniformly eventually positive semigroup on $\ell^2$ that is irreducible but not persistently irreducible.}
	
	Let $(r_n)$ be the orthonormal basis of the Hilbert space $L^2(0,1)$ that consists of the Rademacher functions and let $U:L^2(0,1)\to \ell^2$ be the unitary operator that maps each function to its coefficients with respect to the basis $(r_n)$.
    Let $(e^{tB})_{t\ge 0}$ denote the left shift semigroup on $L^2(0,1)$ (that is nilpotent).
    We show that the semigroup on $\ell^2$ given by $e^{tA}=Ue^{tB}U^{-1}$ for each $t \ge 0$ is irreducible. 
    However, the semigroup is clearly not persistently irreducible as it is nilpotent.
    
    In order to show that $(e^{tA})_{t\ge 0}$ is irreducible, let $I$ be a non-zero closed ideal of $\ell^2$ that is invariant under the semigroup $(e^{tA})_{t\ge 0}$. Then there exists $k\in\bbN$ such that $e_k\in I$; here $(e_n)$ denotes the standard orthonormal basis of $\ell^2$. 
    For every index $j\ne k$ we have
	\begin{align*}
	    \duality{e^{tA}e_k}{e_j} 
	    &= \duality{e^{tB}r_k}{r_j};
	\end{align*}
	and the term on the right is non-zero for some $t\in [0,1]$ (for instance, for all $t<1$ that are sufficiently close to $1$). 
    So for this time $t$ the vector $\modulus{e^{tA}e_k}$ dominates a non-zero multiple of $e_j$. 
    As $\modulus{e^{tA}e_k}$ is an $I$, so is $e_j$. 
    Thus, $I=\ell^2$.
\end{example}

While the above counterexample shows that Condition~\ref{con:irreducibility}\ref{con:irreducibility:itm:irred} does not imply~\ref{con:irreducibility:itm:pers-irred} even in the case of uniformly eventually positive semigroups, we do not know whether the semigroup in the example satisfies~\ref{con:irreducibility:itm:weak-arbi-t}. 
So it remains open whether any of the implications ``\ref{con:irreducibility:itm:irred} $\Rightarrow$ \ref{con:irreducibility:itm:weak-arbi-t}'' or ``\ref{con:irreducibility:itm:weak-arbi-t} $\Rightarrow$ \ref{con:irreducibility:itm:pers-irred}'' is true for (individually or uniformly) eventually positive semigroups 
(but note that they cannot both be true, as Example~\ref{exa:rademacher} shows).

Finally, let us briefly consider the case of analytic semigroups.  
This case is simpler since a phenomenon as in Example~\ref{exa:rademacher} cannot occur:

\begin{proposition}
    \label{prop:irred-analytic}
    Let $E$ be complex Banach lattice and assume that $(e^{tA})_{t \ge 0}$ is an individually eventually positive $C_0$-semigroup on $E$. 

    If the semigroup is analytic, then all five Conditions~\ref{con:irreducibility} are equivalent.
\end{proposition}

\begin{proof}
    The remaining implication ``\ref{con:irreducibility:itm:irred}~$\Rightarrow$~\ref{con:irreducibility:itm:pers-irred}'' in Theorem~\ref{thm:irreducibility-eventually-positive} follows from the identity theorem for analytic functions.
\end{proof}

For positive semigroups, irreducibility together with analyticity implies a stronger version of positivity, namely that the semigroup operators map all vectors $f \gneq 0$ to quasi-interior points \cite[Theorem~C-III-3.2(b) on p.\,306]{Nagel1986}.  
In the following proposition, we slightly modify the argument to show that the same remains true if the semigroup is only uniformly eventually positive rather than positive. 
We do not know whether individual eventual positivity suffices for the same conclusion.

\begin{proposition}
    \label{prop:irred-analytic-strongly}
    Let $E$ be a complex Banach lattice and assume that $(e^{tA})_{t \ge 0}$ is a uniformly eventually positive analytic $C_0$-semigroup on $E$ and choose $t_0 \in [0,\infty)$ such that $e^{tA} \ge 0$ for all $t \ge t_0$. 

    If  $(e^{tA})_{t \ge 0}$ is (persistently) irreducible, then for every $0\lneq f\in E$ and all $t > 2 t_0$ the vector $e^{tA}f$ is a quasi-interior point of $E_+$.
\end{proposition}

In the light of the terminology of earlier papers on eventual positivity such as \cite{DanersGlueckKennedy2016a, DanersGlueckKennedy2016b} it is natural to refer to the property in the conclusion of the proposition as \emph{uniform eventual strong positivity} of $(e^{tA})_{t \ge 0}$, cf.~\eqref{eq:eventual-strong-positivity}.
In \cite{GregorioMugnolo2020a, GregorioMugnolo2020b} this property was called \emph{eventual irreducibility}.

\begin{proof}[Proof of Proposition~\ref{prop:irred-analytic-strongly}]
    We first make the following preliminary observation: 
    
    $(*)$ 
    If the orbit of a vector $g \in E_+$ is positive (meaning that $e^{tA}g \ge 0$ for all $t \ge 0$) and $(t_n) \subseteq (0,\infty)$ converges to $0$ sufficiently fast, then we can find an increasing sequence $(g_n) \in E_+$ converging to $g$, that satisfies $0 \le g_n \le e^{t_n A} g$ for each index $n$.
    
    Indeed, let $(t_n)$ converge to $0$ so fast that $\sum_{n=1}^\infty \norm{e^{t_nA}g - g} < \infty$. 
    Then we define the (not yet positive) vectors
    \begin{align*}
        g_n 
        := 
        g - \sum_{k=n}^\infty \Big(g - e^{t_k A}g \Big)^+
    \end{align*}
    for each $n \in \bbN$, where the series converges absolutely in $E$.
    Clearly, $(g_n)$ is an increasing sequence of real vectors in $E$ that converges to $g$. 
    For each index $n$ one has $g_n \le g - (g - e^{t_n A}g )^+ = g \land e^{t_nA} g \le e^{t_nA}g$. 
    Since all the vectors $e^{t_nA}g$ are positive we can now replace each $g_n$ with $g_n^+$ to obtain a sequence $(g_n)$ with the desired properties. 
    This proves~$(*)$.

    Assume now that the conclusion of the proposition does not hold, i.e., that there exists a vector $0 \lneq f \in E$ and a time $\tau > 2t_0$ such that $e^{\tau A}$ is not a quasi-interior point of $E_+$. 
    By the characterization of quasi-interior points in \cite[Theorem~II.6.3 on p.\,98]{Schaefer1974}), there is a functional $0 \lneq \varphi \in E'$ such that $\langle \varphi, e^{\tau A} f \rangle = 0$.

    The orbit of the vector $g := e^{t_0 A} f \in E_+$ is positive, so we can apply the preliminary observation~$(*)$ to $g$. 
    Let $(t_n)$ and $(g_n)$ be as given by~$(*)$. 
    By dropping finitely many elements of these sequences we can achieve that $\tau - t_n \ge 2t_0$ for all $n$ and hence, all the operators $e^{(\tau-t_0-t_n)A}$ are positive.
    For all integers $n \ge m \ge 1$ we thus have 
    \begin{align*}
        0 
        \le 
        e^{(\tau-t_0-t_n)A} g_m 
        \le 
        e^{(\tau-t_0-t_n)A} g_n 
        \le 
        e^{(\tau-t_0)A} g 
        = 
        e^{\tau A}f
    \end{align*}
    and thus, $\langle \varphi, e^{(\tau-t_0-t_n)A} g_m \rangle = 0$. 
    As the sequence $(\tau-t_0-t_n)_{n \ge m}$ accumulates at the point $\tau-t_0 \in (0,\infty)$, analyticity of the semigroup implies that $\langle \varphi, e^{tA} g_m \rangle = 0$ for all $t \in [0,\infty)$ and all $m \in \bbN$. As $g_m \to g$, we even have 
    $0 = \langle \varphi, e^{tA} g \rangle = \langle \varphi, e^{(t_0+t)A}f \rangle$ for all $t \in [0,\infty)$.
    According to Theorem~\ref{thm:irreducibility-eventually-positive} this contradicts the persistent irreducibility.
\end{proof}

In Example~\ref{exa:coupling-irreducible-not-strong}, we show that the assumption of analyticity cannot be dropped in Proposition~\ref{prop:irred-analytic-strongly}.

\section{Spectral properties of persistently irreducible semigroups}
    \label{sec:spectrum}

In this section, we study spectral properties of eventually positive semigroups that are persistently irreducible. 
For the case of positive irreducible semigroups, most of these properties are proved in \cite[Section~C-III.3]{Nagel1986}. 
It is instructive to observe that the conclusions of several results in this section resemble similar properties that were shown under stronger conditions in \cite[Section~5]{DanersGlueck2017} (the conclusions are formulated somewhat differently in \cite[Section~5]{DanersGlueck2017}, but can be rephrased in terms of leading eigenvectors, see \cite[Proposition~3.1 and Corollary~3.3]{DanersGlueckKennedy2016b}).

Recall that a linear positive functional $\varphi$ on a Banach lattice is said to be \emph{strictly positive} if its kernel contains no positive non-zero element.

\begin{proposition}
	\label{prop:strong-positivity-of-eigenvectors}
	Let $E$ be a complex Banach lattice and let $(e^{tA})_{t \ge 0}$ be an individually eventually positive and persistently irreducible $C_0$-semigroup on $E$.
	\begin{enumerate}[label=\upshape(\alph*), leftmargin=*, widest=iii]
		\item 
		If $0 \lneq u \in E$ is an eigenvector of $A$ for an eigenvalue $\lambda \in \bbR$, then $u$ is a quasi-interior point of $E_+$.
		
		\item 
		If $0 \lneq \psi \in E'$ is an eigenvector of $A'$ for an eigenvalue $\lambda \in \bbR$, then $\psi$ is strictly positive.
	\end{enumerate}
\end{proposition}

\begin{proof}
    (a)
    Let $0 \lneq \varphi \in E'_+$. 
    According to Theorem~\ref{thm:irreducibility-eventually-positive}, Condition~\ref{con:irreducibility}\ref{con:irreducibility:itm:weak-arbi-t} is satisfied, so we can find $t \in [0,\infty)$ such that 
    \begin{align*}
	0 < \langle \varphi, e^{tA} u \rangle = \langle \varphi, e^{t \lambda} u \rangle = e^{t \lambda} \langle \varphi, u \rangle,
    \end{align*}
    so $\langle \varphi, u \rangle > 0$. 
    This shows that $u$ is a quasi-interior point \cite[Theorem~II.6.3 on p.\,98]{Schaefer1974}.
    
    (b) 
    This follows by a similar argument as~(a).
\end{proof}

For the following theorem, recall that an eigenvalue $\lambda$ of a linear operator $A: E \supseteq \dom{A} \to E$ on a Banach space $E$ is called \emph{geometrically simple} if the eigenspace $\ker(\lambda-A)$ is one-dimensional. 
It is called \emph{algebraically simple} if the generalized eigenspace $\bigcup_{n \in \bbN} \ker\big( (\lambda - A)^n \big)$ is one-dimensional. 

\begin{theorem}
	\label{thm:simple-eigenvalue}
	Let $(e^{tA})_{t \ge 0}$ be a real, individually eventually positive, and persistently irreducible $C_0$-semigroup on a complex Banach lattice $E$.
	
	If $\lambda \in \bbR$ is an eigenvalue of both $A$ and $A'$ and $\ker(\lambda - A')$ contains a positive non-zero functional, then $\lambda$ is algebraically simple as an eigenvalue of $A$, the eigenspace $\ker(\lambda - A)$ is spanned by a quasi-interior point of $E_+$, and $\ker(\lambda - A')$ contains a strictly positive functional.
\end{theorem}

For a proof of Theorem~\ref{thm:simple-eigenvalue}, we need the following general result:

\begin{lemma}
    \label{lem:geometric-simplicity-plus-non-zero-test}
    Let $A: X \supseteq \dom{A} \to X$ be a closed and densely defined linear operator on a complex Banach space $E$ and let $\lambda \in \bbC$ be an eigenvalue of both $A$ and $A'$. 
    If $\lambda$ is geometrically simple as an eigenvalue of $A$ and there exist eigenvectors $u \in \ker(\lambda - A)$ and $\varphi \in \ker(\lambda - A')$ such that $\langle \varphi, u \rangle \not= 0$, then $\lambda$ is also algebraically simple as an eigenvalue of $A$.
\end{lemma}

It seems likely that Lemma~\ref{lem:geometric-simplicity-plus-non-zero-test} is known to experts in spectral theory, although we could not find an explicit reference for it. 
For matrices, the lemma is implicitly shown in the proof of \cite[Theorem~7]{Cairns2021}, whereas the infinite-dimensional version is implicitly shown in the proof of (iii) implies (iv) of \cite[Proposition~3.1]{DanersGlueckKennedy2016b}. 

\begin{proof}[Proof of Lemma~\ref{lem:geometric-simplicity-plus-non-zero-test}]
	Without loss of generality, we may assume that $\lambda = 0$.
	Let $v \in \ker A^2$; it suffices to prove that $v \in \ker A$. 
	Since $Av \in \ker A$ and $\ker A$ is spanned by $u$, there exists a scalar $\alpha$ such that $Av = \alpha u$. 
	By testing this equality against $\varphi$ and using that $A'\varphi = 0$, we obtain
	\begin{align*}
		0 = \langle \varphi, Av \rangle = \alpha \langle \varphi, u \rangle.
	\end{align*}
	Since $\langle \varphi, u \rangle \not= 0$, this implies that $\alpha = 0$, so $Av = 0$.
\end{proof}

\begin{proof}[Proof of Theorem~\ref{thm:simple-eigenvalue}]
	There is no loss of generality in assuming that $\lambda = 0$.
	
	Let $0 \lneq \varphi \in \ker(A')$.
	According to Proposition~\ref{prop:strong-positivity-of-eigenvectors}, $\varphi$ is strictly positive and all positive non-zero elements of $\ker(A)$ are quasi-interior points of $E_+$.
	
	Now we show that the real part  $E_\bbR \cap \ker(A)$ of $\ker(A)$ is a sublattice of $E_\bbR$.
	Let $f \in E_\bbR \cap \ker(A)$; it suffices to show that $\modulus{f} \in \ker(A)$. 
	Due to the individual eventual positivity of the semigroup, there exists $t_0 \ge 0$ such that
	\begin{align*}
		\modulus{f} = \modulus{e^{tA}f} \le e^{tA} \modulus{f}
	\end{align*}
	for all $t \ge t_0$; this is a general property of individually eventually positive operator nets, see \cite[Lemma~4.2]{DanersGlueck2017}. 
	For $t \ge t_0$ the vector $e^{tA} \modulus{f} - \modulus{f}$ is therefore positive; 
	but it is also in the kernel of $\varphi$, and hence equal to $0$ due to the strict positivity of $\varphi$. 
	Thus, $e^{tA}\modulus{f} = \modulus{f}$ for all $t \ge t_0$.
	For general $t \ge 0$ this implies
	\begin{align*}
		e^{tA}\modulus{f} = e^{tA} e^{t_0 A} \modulus{f} = e^{(t+t_0)A} \modulus{f} = \modulus{f},
	\end{align*}
	so $\modulus{f} \in \ker A$, as claimed.
	
	Now we can show that $\ker A \cap E_\bbR$ is one-dimensional and spanned by a quasi-interior point of $E_+$.
	We have just seen that $\ker A \cap E_\bbR$ is a closed sublattice of the real Banach lattice $E_\bbR$.
	Since every non-zero positive vector in $\ker A \cap E_\bbR$ is a quasi-interior point within the Banach lattice $E_\bbR$, it is also a quasi-interior point within the Banach lattice $\ker A \cap E_\bbR$, see \cite[Corollary~2 to Theorem~II.6.3 on p.\,98]{Schaefer1974}. 
	Hence, $\ker A \cap E_\bbR$ is at most one-dimensional \cite[Remark~5.9]{Glueck2018}. 
	Since $A$ is real and $\ker A$ is non-zero, we conclude that $\ker A \cap E_\bbR$ is one-dimensional and thus spanned by a vector $u \not= 0$.
    The modulus $\modulus{u}$ is also a non-zero vector in $\ker A \cap E_\bbR$ and thus spans this space, too. 
    By what we have noted at the beginning of the proof, $\modulus{u}$ is a quasi-interior point of $E_+$.
    The fact that $A$ is real implies that $\ker A$ is also spanned by $\modulus{u}$ (over $\bbC$).
	
	Finally, we note that the eigenvalue $\lambda = 0$ of $A$ is algebraically simple. 
	This follows from the geometric simplicity and $\langle \varphi, u \rangle > 0$ by means of Lemma~\ref{lem:geometric-simplicity-plus-non-zero-test}.
\end{proof}

In the following corollary, we list a few simple consequences of Theorem~\ref{thm:simple-eigenvalue}. 
If $E$ is a Banach space, then for each $u\in E$ and $\varphi \in E'$, the (at most) rank-one operator $u\otimes \varphi$ is defined as
\begin{align*}
    u \otimes \varphi: 
    \; 
    E \to E,
    \qquad 
    f & \mapsto \duality{\varphi}{f}u.
\end{align*}

\begin{corollary}
	\label{cor:simple-eigenvalue}
	Let $E$ be a complex Banach lattice and let $(e^{tA})_{t \ge 0}$ be a real, individually eventually positive, and persistently irreducible $C_0$-semigroup on the space $E$.
	\begin{enumerate}[label=\upshape(\alph*), leftmargin=*, widest=iii]
		\item 
            Assume that the spectral bound $\spb(A)$ is not $-\infty$ and is an eigenvalue of $A$. 
            If the operator family $\big( (\lambda - \spb(A)) \Res(\lambda,A) \big)_{\lambda > \spb(A)}$ is bounded in some right neighbourhood of $\spb(A)$,
            then $\spb(A)$ is an algebraically simply eigenvalue of $A$, the eigenspace $\ker(\spb(A) - A)$ is spanned by a quasi-interior point of $E_+$, and the dual eigenspace $\ker(\spb(A) - A')$ contains a strictly positive functional.
		
		\item 
		If the semigroup is mean ergodic and the mean ergodic projection $P$ is non-zero, then $P = u \otimes \varphi$ for a quasi-interior point $u$ of $E_+$ and a strictly positive functional $\varphi \in E'$.
		
		\item 
		If $\spb(A)$ is not $-\infty$ and is a pole of the resolvent, then the pole is simple and the corresponding spectral projection $P$ is given by $P=u\otimes \varphi$ for a quasi-interior point $u$ of $E_+$ and a strictly positive functional $\varphi \in E'$.
	\end{enumerate}
\end{corollary}

The conclusion of part~(c) of the corollary implies, in particular, that the spectral projection $P$ has finite rank and hence, 
$\spb(A)$ is a so-called \emph{Riesz point} of $A$.

\begin{proof}[Proof of Corollary~\ref{cor:simple-eigenvalue}]
	(a) 
        We may assume that $\spb(A) = 0$. 
        According to Theorem~\ref{thm:simple-eigenvalue} it suffices to show that $\ker A'$ contains a non-zero positive element. 
	Due to the boundedness assumption on the resolvent, $\ker A'$ contains a non-zero element $\varphi$ \cite[Proposition~4.3.6]{ArendtBattyHieberNeubrander2011}. 
        The proof in this reference also shows how $\varphi$ can be obtained: 
        let $u \in \ker A$ be non-zero and choose an abritrary functional $\psi \in E'$ such that $\langle \psi, u \rangle \not= 0$. 
        Then the net $\big( \lambda \Res(\lambda,A)'\psi \big)_{\lambda \in (0,\infty)}$ -- where $(0,\infty)$ is ordered conversely to the order inherited from $\bbR$ --  
        has a weak${}^*$-convergent subnet by the Banach--Alaoglu theorem and the limit $\varphi$ of this subnet is non-zero and an element of $\ker A'$. 
        
        In this argument, we may choose the initial functional $\psi$ to be positive.
        The individual eventual positivity of the semigroup implies that, for each $0\le f\in E$, the distance of $\lambda \Res(\lambda,A)f$ to $E_+$ converges to $0$ as $t \to \infty$ \cite[Corollary~7.3]{DanersGlueckKennedy2016a}.
        Thus, it follows from the positivity of $\psi$ that $\varphi$ is positive, as well.
        
	(b) 
        If $(e^{tA})_{t \ge 0}$ is mean ergodic, the mean ergodic projection $P$ is the projection onto $\ker A$ along $\overline{\Ima A}$ and $\ker A$ separates $\ker A'$; see \cite[Proposition~4.3.4(a)]{ArendtBattyHieberNeubrander2011}. 
        The same reference also shows that $(0,\infty)$ is in the resolvent set of $A$ and that $\big(\lambda \Res(\lambda, A)\big)_{\lambda \in (0,\infty)}$ is bounded in a right neighbourhood of $0$.
        
        Since $\Ima P = \ker A$ and $P$ was assumed to be non-zero, $\ker A$ contains a non-zero element. 
        So $\spb(A) = 0$ and we can apply~(a) to conclude that $\ker A$ is spanned by a quasi-interior point $u$ of $E_+$ and that $\ker A'$ contains a strictly positive functional $\varphi$.
        
        Being a projection, $P$ satisfies $\dim \Ima P' = \dim \Ima P$, see for instance \cite[Lemma~1.2.6]{Glueck2016}, 
        and $\dim \Ima P = \dim \ker A$ is one according to the preceding paragraph. 
        Moreover, as $P'$ is the weak${}^*$-limit of the dual operators of the Cesàro means of $(e^{tA})_{t \ge 0}$, its range $\Ima P'$ contains $\ker A'$ and thus, in particular, $\varphi$, so $\Ima P'$ is actually spanned by $\varphi$.
        Re-scaling $\varphi$ and $u$ such that $\duality{\varphi}{u}=1$, we deduce that $P = u \otimes \varphi$.
	
	(c) 
	Since $\spb(A)$ is a pole of the resolvent $\Res(\argument,A)$, it is an eigenvalue of both $A$ and $A'$ and the corresponding eigenspaces each contain a positive, non-zero vector; see \cite[Proposition~6.2.7]{Glueck2016}. Therefore, by Theorem~\ref{thm:simple-eigenvalue}, the spectral bound $\spb(A)$ is an algebraically simple eigenvalue of $A$, the eigenspace $\ker(\spb(A)-A)$ is spanned by a quasi-interior point $u$ of $E_+$,  and $\ker(\spb(A)-A')$ contains a strictly positive functional $\varphi$. 
    It follows from the algebraic simplicity that the pole $\spb(A)$ of the resolvent of $A$ is simple, see for instance, \cite[Proposition~A.3.2(d)]{Glueck2016}. In particular, $\Ima P=\ker(\spb(A)-A)$ and $\Ima P'=\ker(\spb(A)-A')$. 
    Since $\Ima P$ and $\Ima P'$ have the same dimension \cite[Lemma~1.2.6]{Glueck2016}, the image $\Ima P'$ is spanned by $\varphi$. 
    By re-scaling $\varphi$ and $u$ such that $\duality{\varphi}{u}=1$, we get $P = u \otimes \varphi$.
\end{proof}

After the preceding results on eigenvectors we discuss in Theorem~\ref{thm:non-empty-spec-irred_C_0} below that, on spaces of continuous functions, persistent irreducibility gives that the spectrum of the semigroup generator is non-empty.
For positive semigroups, this is known \cite[Proposition~B-III-3.5(a) on p.\,185]{Nagel1986}, but the proof given in this reference does not directly extend to the eventually positive case since the resolvent of an eventually positive semigroup need not be positive at any point.

Nevertheless, as in the positive case, an essential ingredient for the proof is the observation in the following proposition. 
We say that a bounded linear operator $T$ on a Banach lattice $E$ has \emph{individually eventually positive powers}, if for each $f \in E_+$, there exists an integer $n_0 \ge 0$ such that $T^nf \ge 0$ for each $n \ge n_0$.

\begin{proposition}
    \label{prop:spr}
    Let $T$ be a bounded linear operator on a Banach lattice $E$ and assume that $T$ has individually eventually positive powers. 
    Assume that there is a vector $h \in E_+$ and a number $\delta > 0$ such that $Th \ge \delta h$ and  $T^n h \not= 0$ for each $n \in \bbN$ (so in particular, $h$ is non-zero).
    Then $\spr(T) \ge \delta$.
\end{proposition}

Before the proof it is worthwhile to observe that the conclusion is very easy to see if $T$ is positive: for each $n \in \bbN$ one then has $T^n h \ge \delta^n h$ and thus $\norm{T^n} \ge \delta^n$ as $h \not= 0$. 
In the eventually positive case one can argue as follows:

\begin{proof}[Proof of Proposition~\ref{prop:spr}]
    We may assume that $\norm{h} = 1$.
    Since $Th - \delta h$ is positive, there exists $n_0 \ge 0$ such that $T^{n+1} h \ge \delta T^n h \ge 0$ for each $n \ge n_0$. 
    For every $k \ge 0$ this yields
    \begin{align*}
        T^{n_0+k} h
        \ge 
        \delta T^{n_0+k-1}h 
        \ge 
        \delta^2 T^{n_0+k-2} h
        \ge 
        \dots 
        \ge 
        \delta^{k} T^{n_0} h
        \ge 
        0
        .
    \end{align*} 
    Hence, 
    \begin{align*}
        \norm{T^{n_0+k}} 
        \ge 
        \norm{T^{n_0+k}h}
        \ge 
        \delta^k \alpha,
    \end{align*}
    where $\alpha := \norm{T^{n_0} h}$ is non-zero by assumption. 
    So 
    \begin{align*}
        \spr(T) 
        = 
        \lim_{k \to \infty} 
        \norm{T^{n_0+k}}^{\frac{1}{n_0+k}} 
        \ge 
        \lim_{k \to \infty} 
        \delta^{\frac{k}{n_0+k}} \alpha^{\frac{1}{n_0+k}} 
        = 
        \delta,
    \end{align*}
    which proves the proposition.
\end{proof}

Before we use this proposition to study persistently irreducible semigroups on spaces $C_0(L)$ for locally compact $L$, we mention the following simple consequence on $C(K)$-spaces for compact $K$ (i.e., abstractly speaking, on AM-spaces with order unit), which holds without any irreducibility assumption.
This generalizes \cite[Corollary~7.9]{DanersGlueckKennedy2016a}, 
where the semigroup was assumed to be uniformly eventually strongly positive rather than just individually eventually positive. 

\begin{corollary}
    \label{cor:non-empty-spec-C-K}
    Let $(e^{tA})_{t \ge 0}$ be a real and individually eventually positive $C_0$-semigroup on an AM-space $E$ with order unit. 
    If the semigroup is not nilpotent, then $\spec(A)$ is non-empty.
\end{corollary}

\begin{proof}
    Let $\one$ be a (strong) order unit of $E$. 
    Due to the strong continuity at time $0$ and since the semigroup is real, there exists $s > 0$ such that $e^{sA}\one \ge \frac{1}{2} \one$. 
    
    Moreover, one has $e^{tA}\one \not= 0$ for each $t \in [0,\infty)$. 
    Indeed, assume the contrary. 
    Then $e^{tA}\one = 0$ for all sufficiently large $t$, say $t \ge t_0$.
    For every real vector $f \in E$ between $0$ and $\one$ one has $0 \le e^{tA}f \le e^{tA} \one$ for all sufficiently large $t$ due to the individual eventual positivity. 
    Hence, each orbit of the semigroup vanishes eventually. 
    By Proposition~\ref{prop:ind-implies-unif-for-closed-subspaces-absorbing} this implies that the semigroup is nilpotent, which contradicts our assumption.
    
    As the powers of the operator $e^{sA}$ are individually eventually positive, we can apply Proposition~\ref{prop:spr} to conclude that the spectral radius of $e^{sA}$ is non-zero. 
    Thus, the growth bound of the semigroup is not $-\infty$. 
    But for individually eventually positive semigroups on AM-spaces with unit it was shown in \cite[Theorem~7.8(iii)]{DanersGlueckKennedy2016a} that the growth bound coincides with the spectral bound. 
    So the spectrum is indeed non-empty.
\end{proof}

The assumption that the semigroup not be nilpotent is not redundant in the corollary: 
there exist real $C_0$-semigroup on $C\big([0,1]\big)$ that are nilpotent (see for instance \cite[Example~B-III-1.2(c) on p.\,165]{Nagel1986}). Obviously, such a semigroup is (even uniformly) eventually positive and yet has  empty spectrum. 
On the other hand, this cannot happen for a positive $C_0$-semigroup on $C(K)$: for those, it is a classical result that the spectrum of the generator is always non-empty \cite[Theorem~B-III-1.1 on p.\,164]{Nagel1986}).

On the space $C_0(L)$ of continuous functions on a locally compact Hausdorff space $L$ that vanish at infinity, it can happen that a positive semigroup is not nilpotent and still its generator has empty spectrum; see \cite[Example~B-III-1.2(b) on p.\,165]{Nagel1986} for a concrete example where this occurs. 
However, if the semigroup is also irreducible, then it follows again that the generator has non-empty spectrum \cite[Proposition~B-III-3.5(a) on p.\,185]{Nagel1986}.
The following theorem generalizes this result to individually eventually positive semigroups.
Since the relation between individual eventual positivity of the semigroup and the resolvent operators is much subtler than in the positive case, we cannot use the same argument as in \cite[Proposition~B-III-3.5(a) on p.\,185]{Nagel1986}. 
We will thus argue with a sum of finitely many semigroup operators rather than with the resolvent.

\begin{theorem}
    \label{thm:non-empty-spec-irred_C_0}
    Let $L$ be locally compact Hausdorff and let $(e^{tA})_{t \ge 0}$ be a real, individually eventually positive, and persistently irreducible $C_0$-semigroup on $C_0(L)$. 
    Then $\spec(A)$ is non-empty.
\end{theorem}

\begin{proof}
    Consider a positive non-zero vector $h \in C_0(L)$ which has compact support $S$. 
    According to Corollary~\ref{cor:persistently-irred-non-zero}, we have $e^{tA}h \not= 0$ for each $t \in [0,\infty)$. Using the individual eventual positivity, we can
    choose a time $t_0 > 0$ such that $e^{tA}h \gneq 0$ for each $t \ge t_0$. 
    Due to the persistent irreducibility,  Theorem~\ref{thm:irreducibility-eventually-positive} shows that the semigroup satisfies Condition~\ref{con:irreducibility}\ref{con:irreducibility:itm:weak-large-t}.
    Hence, for every $x \in L$ there exists a time $t_x \in [t_0, \infty)$ such that $(e^{t_xA}h)(x) > 0$. 

    For every $x \in L$, we denote the open support of the continuous positive function $e^{t_x A}h$ by $U_x$. 
    Since $x \in U_x$ for each $x \in L$, we have $\bigcup_{x \in S} U_x \supseteq S$. 
    Hence, due to the compactness of $S$ we can find finitely many points $x_1, \dots, x_m \in S$ such that $U_{x_1}$, \dots, $U_{x_m}$ cover $S$.
    
    Now consider the bounded linear operator $T := e^{t_{x_1}A} + \dots + e^{t_{x_m}A}$ on $C_0(L)$. 
    Then $Th \ge 0$ and for each $x \in S$ we have $(Th)(x) > 0$. 
    Again by the compactness of $S$ there exists $\varepsilon > 0$ such that $(Th)(x) \ge \varepsilon$ for each $x \in S$. Since $h$ vanishes outside of $S$, it follows that $Th \ge \delta h$ for some $\delta > 0$. 
    
    We also observe that the operator $T$ has individually eventually positive powers.  
    To see this, let $t_{\min}$ denote the smallest of the numbers $t_{x_1}, \dots, t_{x_m}$; 
    then $t_{\min} > 0$ since we chose $t_0$ to be non-zero. 
    Let $f \in E_+$ and consider $\hat t \in [0,\infty)$ such that $e^{tA}f \ge 0$ for each $t \ge \hat t$. 
    Choose an integer $n_0 \ge 1$ such that $t_{\min} n_0 \ge \hat t$. 
    A brief computation then shows that $T^n f \ge 0$ for every $n \ge n_0$.
    On the other hand, as $e^{tA}h \gneq 0$ for all $t \in [t_0,\infty)$, the same computation shows that $T^nh \gneq 0$ for each $n \in \bbN$.

    So the assumptions of Proposition~\ref{prop:spr} are satisfied and thus $T$ has non-zero spectral radius. 
    Since the spectral radius is subadditive on commuting operators, we conclude that at least one of the operators $e^{t_{x_1}A}$, \dots, $e^{t_{x_m}A}$ has non-zero spectral radius (and hence even each of the operators in the semigroup has non-zero spectral radius). 
    Therefore, the growth bound of the semigroup is not $-\infty$. 
    As the growth bound coincides with the spectral bound for individually eventually positive semigroups on general AM-spaces \cite[Theorem~4]{AroraGlueck2021c}, it follows that $\spec(A)$ is non-empty, as claimed.
\end{proof}

\section{A method to construct eventually positive semigroups}
    \label{sec:perturbations}

It is natural to ask for examples of eventually positive semigroups that are persistently irreducible but do not have the eventual strong positivity property~\eqref{eq:eventual-strong-positivity}. 
Actually, one can easily find positive semigroups that satisfy those conditions -- for instance the shift semigroup on $L^p$ over the complex unit circle for any $p \in [1,\infty)$.
However, it is much less clear how to find non-positive examples with precisely those properties.
The major obstacle is that, for non-positive semigroups, eventual positivity is usually checked by the characterization theorems in \cite[Section~5]{DanersGlueckKennedy2016b} and \cite[Section~3]{DanersGlueck2018b} -- but those theorems already yield eventual strong positivity. 
Thus, those theorems do not lend themselves to identifying eventually positive semigroups that are not eventually strongly positive.

In this section, we will provide a condition that can be used to construct examples.
Our approach relies on perturbation theory. 
It has been demonstrated in \cite{DanersGlueck2018a} that perturbation theory for eventual positivity has a number of serious limitations. 
However, we will see in this section that more can be said if the perturbation is known to interact well with the unperturbed semigroup. 
This allows us to obtain examples of eventual positivity in situations where the leading eigenvalue of the generator is not known and where the eventual positivity is not necessarily strong.
As a result, we can build a semigroup in Example~\ref{exa:coupling-irreducible-not-strong} that has all the desired properties listed above: 
it is eventually positive but not positive, and it is persistently irreducible but not eventually strongly positive.

\subsection{Positive perturbations}

We start with the following perturbation result that follows quite easily from the Dyson--Phillips series expansion of perturbed semigroups.

\begin{proposition}
    \label{prop:domination-by-perturbed-sg}
    Let $(e^{tA})_{t \ge 0}$ be a $C_0$-semigroup on a Banach lattice $E$ and let $B \in \calL(E)$. 
    If $e^{tA} B e^{sA} \geq 0$ for all $s,t \geq 0$, then 
    \begin{align*}
        e^{t(A+B)} - e^{tA} \geq 0
    \end{align*}
    for all $t \geq 0$.
\end{proposition}

Note that under the assumptions of the proposition, $B$ is positive.
From the perspective of eventual positivity, the point of Proposition~\ref{prop:domination-by-perturbed-sg} is that any kind of eventual positivity of $(e^{tA})_{t \geq 0}$ is inherited by the perturbed semigroup $(e^{t(A+B)})_{t \geq 0}$.
Let us also point out that if $A$ is real, then so are all semigroups in the proposition, and hence one can rewrite the conclusion as 
\begin{align*}
    e^{t(A+B)} \geq e^{tA}
\end{align*}
for every $t \geq 0$.

\begin{proof}[Proof of Proposition~\ref{prop:domination-by-perturbed-sg}]
    We use the Dyson-Phillips series representation of the perturbed semigroup \cite[Theorem~III.1.10 on p.\,163]{EngelNagel2000}: 
    define $V_0(t) := e^{tA}$ for each $t \geq 0$ and, inductively, 
    \begin{align*}
        V_{n+1}(t) 
        := 
        \int_0^t e^{(t-s)A} B V_n(s) \dx s
    \end{align*}
    for each $t \geq 0$ and each integer $n \geq 0$, where the integral is to be understood in the strong sense.  
    Then 
    \begin{align*}
        e^{t(A+B)} = \sum_{n=0}^\infty V_n(t)
    \end{align*}
    for each $t \geq 0$, where the series converges absolutely with respect to the operator norm. 
    It follows from the assumption of the proposition that $V_1(t) \geq 0$ for all $t \geq 0$. 
    Since, also due to the assumption, $e^{tA}B \geq 0$ for all $t \geq 0$, one thus obtains inductively that $V_n(t) \geq 0$ for each $t \geq 0$ and each integer $n \geq 1$. 
    Thus, $e^{t(A+B)} - e^{tA} = \sum_{n=1}^\infty V_n(t) \geq 0$ for every $t$, as claimed.
\end{proof}

Proposition~\ref{prop:domination-by-perturbed-sg} gives us the option to construct examples of eventually positive semigroups by easy perturbations, without having precise spectral information about the perturbed operator $A+B$ available.
Let us first demonstrate this in a very simple finite-dimensional example.

\begin{example}
    \label{exa:matrix-third-row-and-column-positive}
    Consider the self-adjoint $3 \times 3$-matrix 
    \begin{align*}
        A := 
        \begin{pmatrix}
             7 & -1 & 3 \\ 
            -1 &  7 & 3 \\ 
             3 &  3 & 3
        \end{pmatrix}
        = 
         U 
        \begin{pmatrix}
            0 &   &   \\ 
              & 8 &   \\ 
              &   & 9 
        \end{pmatrix}
        U^*
        ,
    \end{align*}
    where $U = (u_1 \; u_2 \; u_3)$ is the unitary matrix with the columns 
    \begin{align*}
        u_1 := 
        \frac{1}{\sqrt{6}}
        \begin{pmatrix}
            -1 \\ -1 \\ 2
        \end{pmatrix}
	, 
        \quad 
        u_2 := 
        \frac{1}{\sqrt{2}}
        \begin{pmatrix}
            1 \\ -1 \\ 0
        \end{pmatrix}
        , 
        \quad
        \text{and}
        \quad 
        u_3 := 
        \frac{1}{\sqrt{3}}
        \begin{pmatrix}
            1 \\ 1 \\ 1
        \end{pmatrix}
        .
    \end{align*}
    As the rescaled semigroup $(e^{-9t} e^{tA})_{t \geq 0}$ converges to the rank-$1$ projection $u_3 u_3^*$ as $t \to \infty$, we see that $(e^{tA})_{t \geq 0}$ is (uniformly) eventually strongly positive. 
    (Alternatively, one could also derive this from \cite[Theorem~3.3]{NoutsosTsatsomeros2008}.)
    
    However, one can say more in this concrete example: 
    From the diagonalization of $A$ one obtains by a short computation that 
    \begin{align*}
	\renewcommand*{\arraystretch}{1.3} 
	\newcommand{\entryEight}{
            \frac{8^n}{2}
        }
        \newcommand{\entryNine}{
            \frac{9^n}{3}
        }
        A^n = 
	\begin{pmatrix}
            \phantom{-}\entryEight + \entryNine &          - \entryEight + \entryNine & \entryNine \\ 
            - \entryEight + \entryNine & \phantom{-}\entryEight + \entryNine & \entryNine \\ 
		\entryNine &  \entryNine & \entryNine
		\end{pmatrix}
    \end{align*}
    for every integer $n \ge 1$ (but not for $n=0$, as we dropped the term $0^n$ from the formula).
    So for every $t \ge 0$, the third row and the third column of $e^{tA}$ are positive.

    Now consider the perturbation  
    \begin{align*}
        B := 
        \begin{pmatrix}
            0 &   &   \\
              & 0 &   \\ 
              &   & b
        \end{pmatrix}
    \end{align*}
    for any number $b \geq 0$. 
    Then the matrix $e^{tA} B e^{sA}$ is positive for all $s,t \geq 0$, so it follows from Proposition~\ref{prop:domination-by-perturbed-sg} that the matrix semigroup generated by
    \begin{align*}
        A + B 
        = 
        \begin{pmatrix}
             7 & -1 & 3     \\ 
            -1 &  7 & 3     \\ 
             3 &  3 & 3 + b 
        \end{pmatrix}
    \end{align*}
    is (uniformly) eventually strongly positive for any fixed number $b \geq 0$.
\end{example}

In the above example, note that $B$ does not commute with $A$ if $b \not= 0$, so $A+B$ does not have the same eigenvectors as $A$. 
Hence, we were able to derive information about the eventual positivity of $(e^{t(A+B)})_{t \ge 0}$ without knowing the eigenvectors of the generator.
In Example~\ref{exa:coupling-irreducible-not-strong}, we will use the above matrix semigroup to construct an eventually positive semigroup (in infinite-dimensions) that is persistently irreducible but not eventually strongly positive.
To do so, it is desirable in the situation of Proposition~\ref{prop:domination-by-perturbed-sg}, to have a way to check whether the perturbed semigroup is persistently irreducible. 
The following result is useful for this purpose.

\begin{theorem}
    \label{thm:perturbed-sg-inv}
    Let $(e^{tA})_{t \ge 0}$ be a $C_0$-semigroup on a Banach lattice $E$ and let $B \in \calL(E)$.
    Assume that $e^{tA} B e^{sA} \geq 0$ for all $s,t \geq 0$ 
    and that the semigroup $(e^{tA})_{t \ge 0}$ is individually  eventually positive 
    (hence, so is $(e^{t(A+B)})_{t \ge 0}$ according to Proposition~\ref{prop:domination-by-perturbed-sg}).
    Let $I \subseteq E$ be a closed ideal that is uniformly eventually invariant under the perturbed semigroup $(e^{t(A+B)})_{t \ge 0}$.
    
    Then $I$ is uniformly eventually invariant under both the unperturbed semigroup $(e^{tA})_{t \ge 0}$ and the operator family $(e^{tA}Be^{sA})_{s,t \ge 0}$.
\end{theorem}

By saying that $I$ is uniformly eventually invariant under $(e^{tA}Be^{sA})_{s,t \ge 0}$ we mean that $e^{tA}Be^{sA}I \subseteq I$ for all sufficiently large $s$ and $t$ -- i.e., the index set $[0,\infty) \times [0,\infty)$ is endowed with the \emph{product order} given by $(s_1, t_1) \le (s_2, t_2)$ if and only if $s_1 \le s_2$ and $t_1 \le t_2$.

\begin{proof}[Proof of Theorem~\ref{thm:perturbed-sg-inv}]
    We first show the eventual invariance under $(e^{tA})_{t \ge 0}$: let $0 \le f \in I$. 
    Then it follows from the individual eventual positivity of the unperturbed semigroup and from Proposition~\ref{prop:domination-by-perturbed-sg} that $0 \le e^{tA} f \le e^{t(A+B)}f \in I$ for all sufficiently large $t$ (where the time from which on the first inequality holds, depends on $f$). 
    Hence, $I$ is individually eventually invariant under $(e^{tA})_{t \ge 0}$. 
    As $I$ is closed, Corollary~\ref{cor:ind-implies-unif-for-closed-subspaces} even gives the uniform eventual invariance. 

    To show the (individual and whence also uniform) eventual invariance under the family $(e^{tA}Be^{sA})_{s,t \ge 0}$, choose times $t_0,s_0 \ge 0$ such that $0 \le e^{t(A+B)}I \subseteq I$ for all $t \ge t_0$ and $e^{sA}I \subseteq I$ for all $s \ge s_0$. 
    Let $t \ge t_0$ and $s \ge s_0$. 
    
    First consider a vector $0 \le f \in I$ that is contained in $\dom{A} = \dom{A+B}$.
    Then 
    \begin{align*}
        e^{t(A+B)} e^{sA} f \in I.
    \end{align*}
    Since $f \in \dom{A} = \dom{A+B}$ we can take the derivative either with respect to $s$ or $t$. 
    As $I$ is closed this yields 
    \begin{align*}
        e^{t(A+B)} A e^{sA} f \in I 
        \qquad \text{and} \qquad 
        e^{t(A+B)} (A+B) e^{sA} f \in I. 
    \end{align*}
    Thus, $e^{t(A+B)} B e^{sA}f \in I$.
    By assumption $Be^{sA}f$ and $e^{tA} B e^{sA}f$ are positive, so it follows by means of Proposition~\ref{prop:domination-by-perturbed-sg}, that
    $0 \le e^{tA} B e^{sA}f \le e^{t(A+B)} B e^{sA}f$ and thus, 
    $e^{tA} B e^{sA}f \in I$.

    Finally, consider a general vector $0 \le f \in I$. 
    Choose an ($f$-dependent) time $r_0 \ge s_0$ such that $e^{rA}f \ge 0$ for all $r \ge r_0$ and define vectors
    $f_n := \frac{1}{n} \int_0^{1/n} e^{rA} e^{r_0A} f \dx r$ for all integers $n \ge 1$. 
    Then one has $0 \le f_n \in I \cap \dom{A}$ for each $n$, and thus $e^{tA} B e^{sA} f_n \in I$ by what we have already shown. 
    Since $f_n \to e^{r_0A} f$, it follows that $e^{tA} B e^{(s+r_0)A} f \in I$ for all $t \ge t_0$ and all $s \ge s_0$. 
    This shows that $I$ is individually eventually invariant under the operator family $(e^{tA}Be^{sA})_{s,t \ge 0}$, 
    and the uniform eventual invariance thus follows from Corollary~\ref{cor:ind-implies-unif-for-closed-subspaces}.
\end{proof}

\subsection{Coupling eventually positive semigroups}

A trivial way to construct a new eventually positive semigroup from two given ones is to take their direct sum; 
this construction is arguably not particularly interesting, though, as the new semigroup does not show any kind of interaction between the two original ones. 
In particular, the new semigroup will not be (persistently) irreducible, even if the two original semigroups are.
By employing the simple perturbation result from Proposition~\ref{prop:domination-by-perturbed-sg} and the eventual invariance result from Theorem~\ref{thm:perturbed-sg-inv} we will now demonstrate how a coupling term between the two semigroups can be introduced that preserves both the eventual positivity and the persistent irreducibility.

\begin{corollary}
    \label{cor:coupling-of-sg}
    Let $(e^{tA_1})_{t \geq 0}$ and $(e^{tA_2})_{t \geq 0}$ be $C_0$-semigroups on Banach lattices $E_1$ and $E_2$, respectively. 
    Let $B_{12}: E_2 \to E_1$ and $B_{21}: E_1 \to E_2$ be bounded linear operators
    such that 
    \begin{align*}
        e^{tA_1} B_{12} e^{sA_2}: E_2 \to E_1 
        \qquad \text{and} \qquad 
        e^{tA_2} B_{21} e^{sA_1}: E_1 \to E_2
    \end{align*}
    are positive for all $s,t \geq 0$ 
    (so in particular, $B_{12}$ and $B_{21}$ are positive).
    \begin{enumerate}[label=\upshape(\alph*), leftmargin=*, widest=iii]
        \item
        If both semigroups $(e^{tA_1})_{t \geq 0}$ and $(e^{tA_2})_{t \geq 0}$ are individually/uniformly eventually positive, then so is the semigroup
        $(e^{tC})_{t \ge 0}$ on $E_1 \times E_2$ generated by 
        \begin{align*}
            C :=
            \begin{pmatrix}
                A_1 &     \\ 
                    & A_2 
            \end{pmatrix}
            + 
            \begin{pmatrix}
                       & B_{12} \\ 
                B_{21} &  
            \end{pmatrix}
            .
        \end{align*}
        
        \item 
        If both $(e^{tA_1})_{t \geq 0}$ and $(e^{tA_2})_{t \geq 0}$ are individually eventually positive and persistently irreducible and both $B_{12}$ and $B_{21}$ are non-zero, then $(e^{tC})_{t \ge 0}$ is also persistently irreducible.
    \end{enumerate}
\end{corollary}

\begin{proof}
    (a) 
    Due to Proposition~\ref{prop:domination-by-perturbed-sg}, we have that $e^{tC} - e^{tA_1} \oplus e^{tA_2} \ge 0$ for each $t \ge 0$, from which the assertion readily follows.

    (b) 
    Let $I \subseteq E_1 \times E_2$ be a closed ideal that is uniformly eventually invariant under $(e^{tC})_{t \ge 0}$. 
    Then $I = I_1 \times I_2$ for closed ideals $I_1 \subseteq E_1$ and $I_2 \subseteq E_2$. 
    It follows from Theorem~\ref{thm:perturbed-sg-inv} that we can find a time $t_0 \ge 0$ such that $I = I_1 \times I_2$ is invariant under both $e^{tA_1} \oplus e^{tA_2}$ and
    \begin{align*}
        \begin{pmatrix}
            0                        & e^{tA_1} B_{12} e^{sA_2} \\ 
            e^{tA_2} B_{21} e^{sA_1} & 0
        \end{pmatrix}
    \end{align*}
    for all $s,t \ge t_0$.
    As $(e^{tA_1})_{t \ge 0}$ and $(e^{tA_2})_{t \ge 0}$ are persistently irreducible, it follows that $I_1 = \{0\}$ or $I_1 = E_1$ and likewise for $I_2$.

    So it only remains to check that the two cases (1)~ $I_1 = E_1$, $I_2 = \{0\}$ and (2)~$I_1 = \{0\}$, $I_2 = E_2$ cannot occur. 
    Suppose $I_1 = E_1$. 
    As $B_{21}$ is non-zero, the spaces $E_1,E_2$ are non-zero, so we can choose a vector $0 \lneq f_1 \in I_1$. 
    Moreover, as the dual operator $B_{21}'$ is also non-zero, there exists $0 \lneq \varphi_2 \in E_2'$ such that $B_{21}'\varphi_2 \not= 0$. 
    Due to Theorem~\ref{thm:irreducibility-eventually-positive}, the semigroup $(e^{sA_1})_{s \ge 0}$ satisfies Condition~\ref{con:irreducibility}\ref{con:irreducibility:itm:weak-large-t}, so there exists $s \ge t_0$ such that
    \begin{align*}
        \langle \varphi_2, B_{21} e^{sA_1} f_1 \rangle 
        = 
        \langle B_{21}' \varphi_2, e^{sA_1} f_1 \rangle 
        \ne 
        0
    \end{align*}
    and hence, $B_{21}e^{sA_1} f_1 \not= 0$. 
    By applying Corollary~\ref{cor:persistently-irred-non-zero} to the semigroup $(e^{tA_2})_{t \ge 0}$, we conclude that $e^{t_0 A_2} B_{21}e^{sA_1} f_1 \not= 0$. 
    Due to the choice of $t_0$, we have the inclusion $e^{t_0 A_2} B_{21}e^{sA_1} I_1 \subseteq I_2$, so it follows that $I_2 \not= \{0\}$, as claimed. 

    By swapping the roles of $I_1$ and $I_2$, we see that the case $I_1 = \{0\}$, $I_2 = E_2$ cannot occur, either.
\end{proof}

\begin{remark}
    The situation in Corollary~\ref{cor:coupling-of-sg} can be interpreted in a system theoretic sense: 
    On the state spaces $E_1$ and $E_2$ consider the input-output-systems 
    \begin{align*}
        \dot x_1 & = A_1 x_1 + B_1 u_1, & \dot x_2 & = A_2 x_2 + B_2 u_2, 
        \\ 
        y_1 & = C_1 x_1,  & y_2 & = C_2 x_2,
    \end{align*}
    respectively, 
    where $B_k: U_k \to E_k$ are bounded linear operators defined on Banach spaces $U_k$ and $C_k: E_k \to Y_k$ are bounded linear operators to Banach spaces $Y_k$.
    Here, $u_k: [0,\infty) \to U_k$ are interpreted as input signals and $y_k: [0,\infty) \to Y_k$ are interpreted as output signals for each $k \in \{1,2\}$.

    Now we consider bounded linear operators $G_{21}: Y_1 \to U_2$ and $G_{12}: Y_2 \to U_1$ and couple both systems by setting $u_1 := G_{12} y_2$ and $u_2 := G_{21}y_1$. 
    Thus we obtain the coupled differential equation 
    \begin{align*}
        \begin{pmatrix}
            \dot x_1 \\ 
            \dot x_2 
        \end{pmatrix}
        \begin{pmatrix}
            A_1 &     \\ 
                & A_2 
        \end{pmatrix}
        \begin{pmatrix}
            x_1 \\ 
            x_2
        \end{pmatrix}
        + 
        \begin{pmatrix}
                           & B_1 G_{12} C_2 \\ 
            B_2 G_{21} C_1 &  
        \end{pmatrix}
        \begin{pmatrix}
            x_1 \\ 
            x_2
        \end{pmatrix}
        ,
    \end{align*}
    which leads to the operator described in Example~\ref{cor:coupling-of-sg}(a) if one sets $B_{12} := B_1 G_{12} C_2$ and $B_{21} := B_2 G_{21} C_1$.
\end{remark}

We know from Proposition~\ref{prop:irred-analytic-strongly} that if an irreducible uniformly eventually positive semigroup is analytic, then it is even uniformly eventually strongly positive. In the following, we use Corollary~\ref{cor:coupling-of-sg} to construct a non-positive example which shows that one cannot drop the analyticity assumption in this proposition. 

\begin{example}
    \label{exa:coupling-irreducible-not-strong}
    \emph{A non-positive but uniformly eventually positive semigroup that is persistently irreducible but not individually eventually strongly positive in the sense of~\eqref{eq:eventual-strong-positivity}.} 

    Set $E_1 := \bbC^3$ and $E_2 := L^1(\bbR)$. 
    Let $A_1 \in \bbC^{3 \times 3}$ be the matrix $A$ from Example~\ref{exa:matrix-third-row-and-column-positive}, which generates a (uniformly) eventually strongly positive semigroup according to that example. 
    
    Moreover, we choose a semigroup $(e^{tA_2})_{t \geq 0}$ on $L^1(\bbR)$ as follows: 
    for each $t \in (0,\infty)$ let $k_t \in L^1(\bbR)$ be the density function of the Gamma distribution whose mean and variance are both equal to $t$, i.e., $k_t(x) = \frac{1}{\Gamma(t)} x^{t-1} e^{-x}$ for $x \in (0,\infty)$ and $k_t(x) = 0$ for $x \in (-\infty,0]$.
    Then $k_s \star k_t = k_{s+t}$ for all $s,t > 0$ and the convolution with $k_t$ defines a positive $C_0$-semigroup on $L^1(\bbR)$.
    Additionally, for every $t \ge 0$, let $L_t: L^1(\bbR) \to L^1(\bbR)$ be the left shift by $t$. 
    Since each $L_t$ commutes with convolutions, we obtain a $C_0$-semigroup $(e^{tA_2})_{t \ge 0}$ on $L^1(\bbR)$ by setting 
    \begin{align*}
        e^{tA_2} f := L_t (k_t \star f)
    \end{align*}
    for all $t > 0$ and $f \in L^1(\bbR)$ (and $e^{0A_2} := \id$).
    This semigroup is clearly positive and it is (persistently) irreducible for the following reason: 
    let $0 \le f \in L^1(\bbR)$ be non-zero and choose $x_0 \in \bbR$ such that $f$ has non-zero integral over $(-\infty, x_0)$. 
    Then the strict positivity of $k_t$ on $(0,\infty)$ implies that $k_t \star f$ is strictly positive on $[x_0,\infty)$ for every $t > 0$. 
    Hence, $e^{tA_2}f$ is strictly positive on $[x_0-t,\infty)$ for every $t > 0$. 
    So Condition~\ref{con:irreducibility}\ref{con:irreducibility:itm:weak-arbi-t} is satisfied and hence, persistent irreducibility of $(e^{tA_2})_{t \ge 0}$ follows by Proposition~\ref{prop:irreducibility-positive}.

    We now couple the $C_0$-semigroups $(e^{tA_1})_{t \ge 0}$ and $(e^{tA_2})_{t \ge 0}$ to obtain a semigroup on $E_1 \times E_2 = \bbC^3 \times L^1(\bbR)$ as described in Corollary~\ref{cor:coupling-of-sg}(a): 
    choose the operator $B_{21}: \bbC^3 \to L^1(\bbR)$ as $B_{21} z = z_3 \one_{[1,2]}$ for each $z= (z_1,z_2,z_3) \in \bbC^3$ 
    and $B_{12}: L^1(\bbR) \to \bbC^3$ as $B_{12}f = \int_{[-2,-1]} f(x) \dx x \; e_3$ for each $f \in L^1(\bbR)$; where $e_3 \in \bbC^3$ denotes the third canonical unit vector.

    Positivity of the semigroup $(e^{tA_2})_{t \ge 0}$ and of the third row and column of $e^{tA_1}$ for each $t > 0$ (see Example~\ref{exa:matrix-third-row-and-column-positive}) yields that the assumptions of Corollary~\ref{cor:coupling-of-sg} are satisfied. 
    Hence, part~(a) of the corollary tells us that the semigroup $(e^{tC})_{t \ge 0}$ on $E_1 \times E_2$ generated by the operator
    \begin{align*}
        C :=
        \begin{pmatrix}
            A_1 &     \\ 
                 & A_2 
        \end{pmatrix}
        + 
        \begin{pmatrix}
                     & B_{12} \\ 
            B_{21} &  
        \end{pmatrix}
    \end{align*}
    is uniformly eventually positive. 

    However, the semigroup is not positive. 
    To see this we show that, for each $z\in \bbC^3$, the first component of $e^{tC}(z,0)$ is equal to $e^{tA_1}z$ for all $t < 2$.  
    Indeed, let us use the following notation: 
    for each $n \in \bbN_0$ and each $t \ge 0$ let $V_n(t): E_1 \times E_2 \to E_1 \times E_2$ denote the $n$-th term of the Dyson--Phillips series representation of $(e^{tC})_{t \ge 0}$ that we already used in the proof of Proposition~\ref{prop:domination-by-perturbed-sg} \cite[Theorem~III.1.10 on p.\,163]{EngelNagel2000}. 
    For each $t \in [0,\infty)$, one has 
    \begin{align*}
        V_0(t) 
        \begin{pmatrix}
            z \\ 
            0
        \end{pmatrix}
        = 
        \begin{pmatrix}
            e^{tA_1}z \\ 
            0
        \end{pmatrix} 
        \quad \text{and} \quad 
        V_1(t) 
        \begin{pmatrix}
            z \\ 
            0
        \end{pmatrix}
        = 
        \begin{pmatrix}
            0 \\ 
            \int_0^t e^{(t-s)A_2} B_{21} e^{sA_1}z \dx s
        \end{pmatrix}
        .
    \end{align*}
    Further, for each $t \in [0,2)$ the function $\int_0^t e^{(t-s)A_2} B_{21} e^{sA_1}z \dx s \in L^1(\bbR)$ only lives within the spatial interval $[-1,\infty)$, so it is in the kernel of $B_{12}$. 
    Hence, for all $t \in [0,2)$ one has 
    \begin{align*}
        V_2(t) 
        \begin{pmatrix}
            z \\ 
            0
        \end{pmatrix}
        = 
        \begin{pmatrix}
            0 \\ 
            0
        \end{pmatrix}
        ,
        \quad \text{and thus} \quad 
        V_n(t) 
        \begin{pmatrix}
            z \\ 
            0
        \end{pmatrix}
        = 
        \begin{pmatrix}
            0 \\ 
            0
        \end{pmatrix}
        \quad \text{for all } n \ge 2
        .
    \end{align*}
    So the first component of  $e^{tC} (z, 0)$ is equal to $e^{tA_1} z$ for all $t \in [0,2)$, as claimed.
    As the semigroup $(e^{tA_1})$ is not positive we can choose a vector $0 \le z \in \bbR^3$ such that $e^{tA_1}z \not\ge 0$ for some $t \in (0,2)$, so $(e^{tC})_{t \ge 0}$ is indeed not positive.

    Next, we note that $(e^{tC})_{t \ge 0}$ is persistently irreducible. 
    This follows from Corollary~\ref{cor:coupling-of-sg}(b): 
    the semigroup $(e^{tA_1})_{t \ge 0}$ is eventually strongly positive (see Example~\ref{exa:matrix-third-row-and-column-positive}) and is thus persistently irreducible. 
    The semigroup $(e^{tA_2})_{t \ge 0}$ is, as observed above, also persistently irreducible. 
    As $B_{12}$ and $B_{21}$ are non-zero, part~(b) of Corollary~\ref{cor:coupling-of-sg} is applicable, as claimed. 

    Finally, we show that $(e^{tC})_{t \ge 0}$ is not eventually strongly positive in the sense of~\eqref{eq:eventual-strong-positivity}.
    To this end we use the Dyson--Phillips series representation of $(e^{tC})_{t \ge 0}$ again. 
    It can be checked by induction that, for every $z \in \bbC^3$ and each $t \ge 0$, the second component of $V_n(t)(z,0)$ is supported in the interval $[1-t,\infty)$ for each $n \in \bbN_0$.
    Hence, the same is true for the second component of $e^{tC}(z,0)$ for every $t \ge 0$.
    So, $e^{tC}(z, 0)$ is not a quasi-interior point of $E_1 \times E_2$ for any $t \ge 0$ and any $0 \le z \in \bbR^3$.
\end{example}

\subsection*{Acknowledgements} 

This article is based upon work from COST Action CA18232 MAT-DYN-NET, supported by COST (European Cooperation in Science and Technology).
The article was initiated during a very pleasant visit of both authors at the University of Salerno in Spring'22. 
The first author is indebted to COST Action 18232 and the second author to the Department of Mathematics of the University of Salerno for financial support for this visit.

\bibliographystyle{plainurl}
\bibliography{literature}

\end{document}